\documentclass[hidelinks,onefignum,onetabnum]{siamart220329}
\usepackage{lipsum}
\usepackage{amsfonts}
\usepackage{graphicx}
\usepackage{epstopdf}
\usepackage{algorithm,algorithmic}
\usepackage{graphicx}
\usepackage{epstopdf}
\usepackage{algorithmic}
\usepackage{verbatim}
\usepackage{multirow}
\usepackage[caption=false]{subfig}
\usepackage{amsmath,booktabs,ctable,threeparttable}
\usepackage{amssymb,boxedminipage}
\usepackage{microtype}
\allowdisplaybreaks[4]
\ifpdf
  \DeclareGraphicsExtensions{.eps,.pdf,.png,.jpg}
\else
  \DeclareGraphicsExtensions{.eps}
\fi

\newsiamremark{remark}{Remark}
\newsiamremark{hypothesis}{Hypothesis}
\crefname{hypothesis}{Hypothesis}{Hypotheses}
\newsiamthm{claim}{Claim}
\newsiamremark{example}{Example}

\headers{Nonlinear Rayleigh--Ritz and refined Rayleigh--Ritz}{Z. Jia and Q. Zheng}

\title{An analysis of the Rayleigh--Ritz and refined Rayleigh--Ritz methods for
regular nonlinear eigenvalue problems\thanks{Submitted to the editors DATE.
\funding{The first author was supported by the National Natural Science Foundation of China (NSFC) under
grant 12171273.
The second author was supported by the Science Foundation of China University of Petroleum, Beijing under
grant 2462024BJRC012, the State Key Laboratory of Petroleum Resources and Prospecting, China University of Petroleum, and the National Natural Science Foundation
of China (NSFC) under grant 62372467.}}}

\author{Zhongxiao Jia\thanks{Corresponding author. Department of Mathematical Sciences, Tsinghua University, 100084 Beijing, China
  (\email{jiazx@tsinghua.edu.cn}).}
\and Qingqing Zheng\thanks{College of Science, China University of Petroleum, Beijing 102249, China
  (\email{zhengqq@cup.edu.cn}).}
}

\usepackage{amsopn}
\DeclareMathOperator{\diag}{diag}

\begin{document}

\maketitle

\begin{abstract}
We establish a general convergence theory of the Rayleigh--Ritz method and
the refined Rayleigh--Ritz method
for computing some simple eigenpair $(\lambda_{*},x_{*})$
of a given analytic regular nonlinear eigenvalue problem (NEP).
In terms of the deviation $\varepsilon$ of $x_{*}$
from a given subspace $\mathcal{W}$, we establish a priori
convergence results on the Ritz value, the Ritz vector and the refined Ritz
vector. The results show that, as
$\varepsilon\rightarrow 0$, there exists a Ritz value
that unconditionally converges to $\lambda_*$
and the corresponding refined Ritz vector does so too
but the Ritz vector converges conditionally and it may fail to converge and even may not be unique.
We also present an error bound for the approximate eigenvector
in terms of the computable residual norm of a given approximate
eigenpair, and give lower and upper bounds
for the error of the refined Ritz vector and
the Ritz vector as well as for that of the corresponding residual norms.
These results nontrivially extend some convergence results
on these two methods for the linear eigenvalue problem to the NEP. Examples
are constructed to illustrate the main results.
\end{abstract}



\section{Introduction}
Given the nonlinear eigenvalue problem (NEP)
\begin{equation}\label{NEP}
    T(\lambda)x=0,
\end{equation}
where $T(\cdot): \Omega\subseteq\mathbb{C}\rightarrow\mathbb{C}^{n\times n}$ is
a nonlinear analytic or holomorphic matrix-valued function in
$\lambda\in\Omega$\ with $\Omega$ a connected open set, assume that
$T(\lambda)$ is regular, that is, $T(\lambda)$ does not have
identically zero determinant for all $\lambda\in \Omega$. Then
the spectrum of $T(\cdot)$ in $\Omega$ is defined by the set \cite{Guttel,Meerbergen}:
\begin{equation}\label{spect}
\Lambda(T(\cdot)):=\bigg\{\lambda\in \Omega\subseteq\mathbb{C}: \text{det}(T(\lambda))=0\bigg\}.
\end{equation}
We call $\lambda\in \Lambda(T(\cdot))$ an eigenvalue of $T(\cdot)$
and $x$ a right eigenvector associated with $\lambda$. Correspondingly,
a vector $y$ satisfying $y^HT(\lambda)=0$ is
called a left eigenvector, where the superscript $H$ denotes the conjugate
transpose of a matrix or vector.
Every eigenvalue $\lambda$ is isolated; that is, there exists
an open neighborhood of $\lambda$ in $\Omega$ so that $\lambda$ is the unique eigenvalue of $T(\cdot)$
in the neighborhood; see \cite{Neumaier} and Theorem 2.1 of \cite{Guttel}.
Throughout the paper, we assume that the right eigenvectors $x$ are scaled to have unit-length, i.e.,
$\|x\|=1$, where $\|\cdot\|$ is the 2-norm of a matrix or vector.

NEP (\ref{NEP}) arises in many applications, such as dynamic analysis of damped
structural systems \cite{app4}, nonlinear ordinary differential eigenvalue
problems \cite{app3}, stability analysis of linear delayed systems \cite{app2},
acoustic problems with absorbing boundary conditions \cite{Day}, electronic
structure calculation for quantum dots \cite{Betcke}, and stability analysis in
fluid mechanics \cite{app1}. More NEPs and details can be found in, e.g.,
\cite{Higham2,Guttel,Higham1,Mehrmann,Mehrmann2001,Porzio,schreiber}
and the references therein.
There are many differences between NEP (\ref{NEP}) and the regular linear eigenvalue problem  $(A-\lambda B)x=0$.
For example, the number of eigenvalues of NEP can be arbitrary, and eigenvectors associated with distinct eigenvalues may be linearly dependent. More theoretical background on NEP can be found in \cite{Guttel}.

Notice that \eqref{NEP} can be written as
\begin{equation}\label{NEP1}
    T(\lambda)=\sum_{i=0}^{k}f_{i}(\lambda)A_{i}
\end{equation}
for some $k\leq n^2$, where $f_{i}: \Omega\subseteq \mathbb{C}\rightarrow
\mathbb{C}$ are nonlinear analytic functions,
$A_{i}\in \mathbb{C}^{n\times n}$ are constant coefficient matrices. The
polynomial eigenvalue problem (PEP) \cite{Przemieniecki,Rothe,Voss0}
corresponds to $f_{i}(\lambda)=\lambda^{i}, \ i=0, 1, \ldots,k$,  which
includes the quadratic eigenvalue problem (QEP), i.e., $k=2$.
The QEP has wide applications and has been intensively studied
\cite{Hilliges,Lancaster,Leguillon,Mackey,Mehrmann2001,tisseur2002}. Another
class of \eqref{NEP1} is the rational eigenvalue problem (REP), where
$f_{i}(\lambda),\ i=0, 1, \ldots,k$ are rational functions of
$\lambda$. Very importantly, from both perspectives of mathematics and
numerical computations, NEP (\ref{NEP}) can be approximated by rational
matrix functions with arbitrary accuracy \cite{dopico,guttel2022,Guttel,Saad}.
As a result, one can solve PEPs and REPs by, e.g.,
linearizing them to standard and generalized linear eigenvalue problems, and
obtains their eigenvalues and eigenvectors.

Several methods have been available for small to medium sized
NEP (\ref{NEP1}). For instance, Newton-type methods \cite{N1,N2,Ruhe,schreiber}
are popular choices for NEPs, and the QR or QZ method is used when NEPs are
transformed into linear eigenvalue problems \cite{dopico,Saad}. In this paper,
we are concerned with the solution of a large sized (\ref{NEP1}).
Suppose that we want to find some {\em simple} eigenpair
$(\lambda_*,x_*)$ of NEP (\ref{NEP1}); that is, $\text{det}(T(\lambda))=0$ has a simple root at $\lambda
=\lambda_*$. For this kind of NEPs, projection methods are most commonly used
numerical methods, which include the nonlinear Lanczos method, the nonlinear
Arnoldi method, nonlinear Jacobi--Davidson type methods, and some others
\cite{Jacobi,Guttel,schreiber,Voss}.

Given a sequence of low $m$-dimensional subspaces $\mathcal{W} \subset \mathbb{C}^n$,
a general projection method projects NEP \eqref{NEP} onto each $\mathcal{W}$ to obtain a
small projected NEP, solving which for an
approximation to the desired $(\lambda_*,x_*)$. The (orthogonal)
Rayleigh--Ritz method with the same left and right projection
subspaces \cite{Guttel} is a widely used projection method for
extracting approximations to
$\lambda_{*}$ and $x_{*}$ with respect to $\mathcal{W}$.
We describe the method as Algorithm \ref{alg-RR}, where suitable approximations
$\mu$ and $\widetilde{x}$, called the Ritz value and Ritz vector,
are selected to approximate $\lambda_*$ and $x_*$, respectively.

\begin{algorithm}
\caption{The Rayleigh--Ritz method}
\label{alg-RR}
   \begin{enumerate}
   \item Compute an orthonormal basis $W$ of $\mathcal{W}$.
    \item  Form $B(\lambda)=W^{H}T(\lambda)W\in \mathbb{C}^{m\times m}$.
    \item   Solve $B(\lambda) z=0$ with $\|z\|=1$, let ($\mu,z$) be an eigenpair of
        $B(\cdot)$ that
        \begin{equation}\label{ritzvalue}
        \mbox{selects\ } \mu\in \Lambda(B(\cdot))\mbox{\ \ based on user's criteria}.
        \end{equation}
    \item   Take the Ritz pair ($\mu,\widetilde{x}$)=($\mu,Wz$) to
        approximate the desired eigenpair ($\lambda_{*},x_{*}$) of $T(\lambda)x=0$.
    \end{enumerate}
\end{algorithm}

It is seen from NEP \eqref{NEP1} that $B(\lambda)$ is naturally analytic in $\lambda\in\Omega$
and can always be formed accurately in exact arithmetic. We may solve the
projected NEP $B(\lambda)z=0$ by a method for
small to medium sized problems, e.g., Newton-type methods or
the more robust QR or QZ algorithm for the matrix or matrix
pencils resulting from PEPs, REPs
themselves and REP approximations of NEPs
\cite{dopico,guttel2022,Saad,vandooren79}.

A few words on the selection strategy \eqref{ritzvalue} of $\mu$. Theoretically,
the Ritz value $\mu$ is selected as the one closest to $\lambda_*$. However,
since $\lambda_*$ is to be sought and unknown, we have to do this by some computationally
viable criteria: If one is interested in some specific $\lambda_*$, e.g.,
the largest one in magnitude, the one with the largest
(or smallest) real part or the one closest to a given target, then we label all the
Ritz values in the corresponding prescribed order, and pick up the Ritz value
$\mu$ as approximation to $\lambda_*$ according to the labeling rule.

We now propose a refined Rayleigh--Ritz method for solving NEP (\ref{NEP}).
For the Ritz value $\mu$, the
method seeks for a unit-length vector $\widehat{x}\in \mathcal{W} $ that
satisfies
\begin{equation}\label{Refin}
    \| T(\mu)\widehat{x}\|=\min_{v\in \mathcal{W},\| v\|=1}\| T(\mu)v\|.
\end{equation}
We call $\widehat{x}$ a refined eigenvector approximation to $x_*$ or
simply the refined Ritz vector corresponding to $\mu$. By definition,
$\widehat{x}=Wy$, where $y$ is the right singular vector of $T(\mu)W$
associated with its smallest singular value. We describe the refined
Rayleigh--Ritz method as Algorithm~\ref{alg-RRR}.

\begin{algorithm}
\caption{The refined Rayleigh--Ritz method}
\label{alg-RRR}
   \begin{enumerate}
   \item Compute an orthonormal basis $W$ of $\mathcal{W}$.
    \item  Form $B(\lambda)=W^{H}T(\lambda)W\in \mathbb{C}^{m\times m}$.
    \item  Solve $B(\lambda) z=0$, and let $\mu$ be the eigenvalue of
    $B(\cdot)$ satisfying \eqref{ritzvalue}.

    \item Solve \eqref{Refin} for $\widehat{x}=Wy$ with $y$ being the right
        singular vector of $T(\mu)W$ associated with its smallest singular
        value, and take the refined Ritz pair ($\mu,\widehat{x}$) to
        approximate ($\lambda_{*},x_{*}$).
    \end{enumerate}
\end{algorithm}

How to construct $\mathcal{W}$ effectively is not our concern in this paper.
For instance, if the residual norm $\|T(\mu)\widetilde x\|$ is not yet below a
prescribed convergence tolerance,
then $\mathcal{W}$ can be expanded
by one step of the Newton iteration with the initial guess $(\mu, \widetilde x)$
\cite{Guttel}. In the current context,
we focus on the convergence of the Ritz value $\mu$, the Ritz vector
$\widetilde{x}$ and the refined Ritz vector $\widehat{x}$
under the assumption that the sequence of $\mathcal{W}$'s contain sufficiently accurate
approximations to the desired eigenvector $x_*$; that is, we suppose
that the deviation
\begin{equation}\label{devia}
    \varepsilon=\sin\angle(x_{*},\mathcal{W})=\|
    W_{\bot}^{H}x_{*}\|=\|(I-P_{\mathcal{W}})x_{*}\|,
\end{equation}
of $x_*$ from $\mathcal{W}$ tends to zero as $\mathcal{W}$ changes,
where the columns of $W_{\bot}$ form an orthonormal basis for the orthogonal
complement of $\mathcal{W}$ in $\mathbb{C}^{n}$ and
$P_{\mathcal{W}}$ is the orthogonal projector onto $\mathcal{W}$.

Since $\sin\angle(x_*,v)\geq\varepsilon$ for any $v\in\mathcal{W}$,
any projection method, e.g., the Rayleigh--Ritz and refined
Rayleigh--Ritz methods, definitely
cannot find an accurate approximation to $x_*$ if $\varepsilon$ is
not small enough. In other words, $\varepsilon\rightarrow 0$ is {\em
necessary} for $\widetilde{x}\rightarrow x_*$ and
$\widehat{x}\rightarrow x_*$. Therefore, in the sequel we always
assume that $\varepsilon\rightarrow 0$, only under which it is meaningful to
speak of the convergence of the Rayleigh--Ritz method and the refined
Rayleigh--Ritz method.

Under the preassumption that $\varepsilon\rightarrow 0$,
this paper focuses on the convergence of the Rayleigh--Ritz method and the refined
Rayleigh--Ritz method for NEP~\eqref{NEP}, and considers the convergence of numerous
sequences of quantities or approximations,
including the Ritz values, Ritz vectors and refined Ritz vectors. As will
be seen, under the preassumption that $\varepsilon\rightarrow 0$, the convergence
guarantees of different sequences may or may not require further condition(s),
meaning that their convergence may differ essentially. To show their possible
differences and present the results clearly, before proceeding,
we introduce the following definition.

\begin{definition}
Under the preassumption that $\varepsilon\rightarrow 0$, if a sequence of quantities or
approximations converge without requiring any further condition,
then the sequence is said to converge unconditionally
or its convergence is unconditional; if its convergence guarantee requires some
further condition(s), then it is said to converge conditionally or its convergence is conditional.
\end{definition}

For the linear eigenvalue problem $(A-\lambda I)x=0$, Jia \cite{Jia1995} and
Jia \& Stewart \cite{Jia1999,Jia2001} prove that as
$\varepsilon\rightarrow 0$, the Rayleigh quotient $W^{H}AW$ has an eigenvalue,
i.e., Ritz value, that converges to $\lambda_{*}$ unconditionally but the Ritz vector converges
conditionally and it may fail to converge
and, even worse, it may not be unique even if
$\varepsilon=0$. The refined Rayleigh--Ritz method extracts a normalized
refined eigenvector approximation $\widehat{x}$ that
minimizes the residual formed with $\mu$ over $\mathcal{W}$ and uses it to
approximate $x_*$ \cite{Jia1997,Jia2000,Jia1999}. It is shown in
\cite{Jia1999,Jia2001} that the convergence of $\widehat{x}$ is unconditional; see also the books
\cite{Stewart2001,vorst2002}. Jia \cite{Jia2004} has made a number of
theoretical comparisons of refined Ritz vectors and Ritz vectors for linear
eigenvalue problems. Huang, Jia and Lin \cite{Huang} generalize the
Rayleigh--Ritz method and the refined Rayleigh--Ritz method to the QEP. By
transforming the QEP into a generalized linear eigenvalue problem, they extend
the main convergence results in \cite{Jia1999,Jia2001} on the Rayleigh--Ritz
method and the refined Rayleigh--Ritz method to the
QEP case.  Hochstenbach and Sleijpen \cite{Hoch} extend the harmonic and
refined Rayleigh--Ritz methods to the PEP. Schreiber \cite{schreiber} and
Schwetlick \& Schreiber
\cite{Schwetlick} consider several nonlinear Rayleigh-type functionals for
NEP (\ref{NEP}), and establish error bounds of the approximate eigenvalues
in terms of the angles of eigenvectors and its approximations.
For NEP (\ref{NEP}) derived from finite element analysis of Maxwell's
equation with waveguide boundary conditions, Liao {\em et al}. \cite{Liao} propose a nonlinear
Rayleigh--Ritz method, but they do not analyze the convergence of Ritz values
and Ritz vectors.

We will study the convergence of Ritz values and Ritz vectors obtained by
Algorithm~\ref{alg-RR} and that of refined Ritz vectors defined by
\eqref{Refin}. We will
extend some of the main convergence results in \cite{Jia2004,Jia1999,Jia2001}
to the NEP case and get insight into similarities and dissimilarities of the
Rayleigh--Ritz method and the refined Rayleigh--Ritz method for the linear
eigenvalue problem and the NEP.

As is expected, all the proofs of the results are highly
nontrivial and have fundamental differences with those for
the linear eigenvalue problem.
We establish all the results under the assumption
that $\lambda_*$ is simple. This
assumption has been dropped in the linear eigenvalue case \cite{Jia2001},
but whether or not it can be dropped in the NEP case
is currently unclear and will be left as future work.
Just as in the linear eigenvalue case,
we will prove that $\mu\rightarrow \lambda_*$ unconditionally
as $\varepsilon\rightarrow 0$, and
present an a priori error estimate for $|\mu-\lambda_*|$
in terms of $\varepsilon$, with the accurate convergence order given.
This error estimate straightforwardly extends to
the eigenvalue perturbation problem of a general nonnormal matrix considered
in \cite{kahan1982}. Our detailed analysis will show that one cannot use
a well-known perturbation bound, i.e.,
\cite[Theorem 8]{kahan1982}, to obtain such error estimate since
it alone does not suffice to reveal the precise convergence order of $\mu$, not as
might have been thought. As will be clear later,
Theorem 8 of \cite{kahan1982} only implies how fast $\mu$ converges {\em at least},
while our error estimate shows how fast it converges with the
precise convergence order determined. Therefore,
regarding the convergence order of $\mu$,
our result is more accurate than the perturbation bound in \cite{kahan1982} indicates.

We will establish an a priori error bound for the Ritz vector $\widetilde{x}$
and shed light on its conditional convergence. We will
show that it may fail to converge to $x_*$ as $\varepsilon\rightarrow 0$.
Even worse, we will see that $\widetilde{x}$ may not be unique even if $\varepsilon=0$;
in the case that there are more than one linearly independent Ritz vectors, each of them
is generally a meaningless approximation to $x_*$. As a consequence,
the residual norm of a selected Ritz approximation may not converge to
zero as $\varepsilon\rightarrow 0$ and even may not be small
even if $\varepsilon=0$. These results
indicate that the Rayleigh--Ritz method may fail to compute
$x_*$ for $\varepsilon\rightarrow 0$.

By bounding the error of an arbitrary approximate eigenvector (not necessarily a Ritz or refined Ritz
vector) in terms of its associated computable residual norm,
we will present an approach to identify the convergence of the
approximate eigenvector. The error bound for the
approximate eigenvector to be obtained
will also play a key role in establishing the convergence results on
the refined Rayleigh--Ritz method. We will prove that $\widehat{x}$ must be unique, provided that
$\varepsilon$ is sufficiently small, which is an important property that
$\widetilde{x}$ does not possess. Furthermore,
we will prove that the refined Ritz vector $\widehat{x}$
converges unconditionally and the residual norm of the refined Ritz pair
must tend to zero as $\varepsilon\rightarrow 0$.
Besides, we will establish lower and upper
bounds for the error of $\widehat{x}$ and $\widetilde{x}$, and study
relationships between the residual norms
$\|T(\mu)\widetilde{x}\|$ and $\|T(\mu)\widehat{x}\|$, showing that
$\|T(\mu)\widehat{x}\|<\|T(\mu)\widetilde{x}\|$ holds strictly and
that one may have $\|T(\mu)\widehat{x}\|\ll \|T(\mu)\widetilde{x}\|$
when the refined Ritz pair $(\mu,\widehat{x})\not=(\lambda_*,x_*)$. This
indicates that $\widehat{x}$ is generally more
accurate and can be much more accurate than $\widetilde{x}$.
We will construct examples to confirm these results for several
$\varepsilon$'s. For each $\varepsilon$, we test 100 random perturbations
of size $O(\varepsilon)$ randomly generated in a normal distribution,
and get more insight into the convergence of the Rayleigh--Ritz method
and the refined Rayleigh--Ritz method.

In section \ref{sec-con-Rv}, we study the convergence of the Ritz value
$\mu$. In sections \ref{sec:R}--\ref{sec:RR},
we consider the convergence of the Ritz vector $\widetilde{x}$ and the
refined Ritz vector $\widehat{x}$, respectively.
In section \ref{sec:bound}, we derive lower and upper bounds
for $\sin\angle(\widetilde{x},\widehat{x})$. We
construct examples to illustrate our results. In section~\ref{sec:concl}, we
conclude the paper with further remarks.

\section{Convergence of the Ritz value}\label{sec-con-Rv}

The following result shows that $\lambda_{*}$
is an exact eigenvalue of some perturbed matrix-valued function
of $B(\cdot)$ in Algorithm~\ref{alg-RR}.

\begin{theorem}\label{l1}
For a given $m$-dimensional subspace $\mathcal{W}$,
let $\varepsilon$ be defined by \eqref{devia}. Then for the projected
matrix-valued function $B(\lambda)=W^HT(\lambda) W$, there exists a
matrix-valued function	$E(\lambda):\lambda\in \Omega\subseteq
\mathbb{C}\rightarrow\mathbb{C}^{m\times m}$ satisfying
	\begin{equation}\label{normE}
		\| E(\lambda_{*})\|\leq\frac{\varepsilon}{\sqrt{1-\varepsilon^{2}}}\| T(\lambda_{*})\|
	\end{equation}
	such that $\lambda_{*}$ is an eigenvalue of the perturbed matrix-valued
function $B(\lambda)+E(\lambda)$.
\end{theorem}

\begin{proof}
Recall from Algorithm~\ref{alg-RR} and \eqref{devia}
that the columns of $W$ and $W_{\bot}$ form orthonormal bases of $\mathcal{W}$
and its orthogonal complement in $\mathbb{C}^n$. Let
$u=W^{H}x_{*}$ and $u_{\bot}=W_{\bot}^{H}x_{*}$. Then
$\|u_{\bot}\|=\varepsilon$. Since
$$
1= \|x_{*}\|^{2}=x_{*}^{H}(WW^{H}+W_{\bot}W_{\bot}^{H})x_{*}=\|u\|^{2}+\|u_{\bot}\|^{2},
$$
we have $\|u\|=\sqrt{1-\varepsilon^{2}}$. From $T(\lambda_{*})x_{*}=0$, we obtain
$$
W^{H}T(\lambda_{*})(WW^{H}+W_{\bot}W_{\bot}^{H})x_{*}=0.
$$
Hence
\begin{equation}\label{Bpert}
	B(\lambda_{*})u+W^{H}T(\lambda_{*})W_{\bot}u_{\bot}=0.
\end{equation}
Let $\widehat{u}=\dfrac{u}{\sqrt{1-\varepsilon^{2}}}$ be a normalized approximate eigenvector
of $B(\mu)z=0$, and denote
\begin{equation}\label{defr}
r=B(\lambda_{*})\widehat{u}.
\end{equation}
Then $r$ is the residual of the approximate
eigenpair $(\lambda_*,\widehat{u})$ of $B(\mu)z=0$.
Define the matrix-valued function
$$
E(\lambda)=\dfrac{1}{\sqrt{1-\varepsilon^{2}}}W^{H}T(\lambda)
W_{\bot}u_{\bot}\widehat{u}^{H}.
$$
Then it is justified from \eqref{Bpert} that
$(B(\lambda_{*})+E(\lambda_{*}))\widehat{u}=0$, which indicates
that $(\lambda_*,\widehat u)$ is an exact eigenpair of
$B(\lambda)+E(\lambda)$. Furthermore,
\begin{eqnarray} 	
\| r\|&=&\|-E(\lambda_*)\widehat{u}\| \nonumber\\
&\leq &\|E(\lambda_*)\|   \nonumber\\
&=&\bigg\|\frac{W^{H}T(\lambda_{*})W_{\bot}u_{\bot}}{\sqrt{1-\varepsilon^{2}}}\bigg\|\nonumber\\
&\leq &\frac{\varepsilon}{\sqrt{1-\varepsilon^{2}}}\|T(\lambda_{*})\|,
\label{err}
\end{eqnarray}
which proves (\ref{normE}).
\end{proof}

Theorem~\ref{l1} extends Theorem 4.1 of \cite{Jia2001} for the linear
eigenvalue problem to the Rayleigh--Ritz method for NEP~\eqref{NEP}.
The theorem shows that $B(\lambda_*)+E(\lambda_*)$ is singular and
$E(\lambda_*)\rightarrow 0$ as $\varepsilon\rightarrow 0$.
Therefore, the smallest singular value $\sigma_{\min}(B(\lambda_{*}))$
of $B(\lambda_{*})$ must tend to zero as $\varepsilon\rightarrow 0$.
The following result quantitatively estimates
how near $B(\lambda_*)$ is to singularity, and will be used later.
\begin{theorem}\label{C1}
The smallest singular value $\sigma_{\min}(B(\lambda_{*}))$ of
$B(\lambda_{*})$ satisfies
\begin{equation}\label{sigmB}	
\sigma_{\min}(B(\lambda_{*}))\leq\frac{\varepsilon}{\sqrt{1-\varepsilon^{2}}}\| T(\lambda_{*})\|.
\end{equation}
\end{theorem}

\begin{proof}
By definition, \eqref{defr} and \eqref{err}, we have
$$\sigma_{\min}(B(\lambda_{*}))=\min_{\|u\|=1}\|B(\lambda_{*})u\|\leq
\|B(\lambda_{*})\widehat{u}\|=\|r\|\leq\frac{\varepsilon}{\sqrt{1-\varepsilon^{2}}}\| T(\lambda_{*})\|.
$$
~
\end{proof}

For $\mu$ defined in \eqref{ritzvalue}, notice that
$\sigma_{\min}(B(\mu))=0$ and $\sigma_{\min}(B(\lambda))$
is continuous in $\lambda\in\Omega$ and that we suppose
$\mu,\lambda_*\in\Omega$.  By the continuity argument,
Theorem~\ref{C1} shows that $\mu\rightarrow \lambda_*$ unconditionally
as $\varepsilon\rightarrow 0$.
Next we establish an a priori upper bound for $|\mu-\lambda_*|$ in
terms of $\varepsilon$, and prove how fast $\mu$ converges to $\lambda_*$
unconditionally. To this end, we first consider the convergence of the Ritz value
$\mu$ when NEP~\eqref{NEP} is a linear one.


\begin{theorem} \label{linearconv}
For the linear eigenvalue problem $T(\lambda_i)x_i=(A-\lambda_i I)x_i=0,\ i=1,2,\ldots,n$ and
$B(\mu_j)z_j=(W^HAW-\mu_j I)z_j=0,\ j=1,2,\ldots,m$,
let the desired $\lambda_*:=\lambda_i$ for a certain $i$ and the approximating
$\mu:=\mu_j$ satisfy \eqref{ritzvalue} for some $j$. Suppose that $\mu\not=\lambda_*$. Then for $\varepsilon$ sufficiently
small, we have the error estimate
\begin{equation}\label{ritzerror}
 |\mu-\lambda_*|=O(\varepsilon^{1/m_{\mu}}),
\end{equation}
where $m_{\mu}$ is the order of the largest Jordan block of $\mu$ in $B(\mu)$.
\end{theorem}

\begin{proof}
For ease of the proof, without loss of generality, for the desired $\lambda_*$ and
the chosen $\mu$, suppose that the Jordan normal
form of $B(\mu)=W^HAW-\mu I$ is $B(\mu)=ZJ(\mu)Z^{-1}$
and that $J(\lambda_*)$ has only one Jordan
block corresponding to such $\mu$, written as
\begin{eqnarray}
J_{\mu}(\lambda_*)&=&
\begin{pmatrix}
\mu-\lambda_* & 1 & &  &\\
 &\mu-\lambda_* &1 & & \\
 & &\ddots &\ddots &\\
 & & & \ddots& 1\\
 & & & &\mu-\lambda_*
\end{pmatrix}
\in \mathbb{C}^{m_{\mu}\times m_{\mu}} \label{jordan}\\
&=:&\diag(\mu-\lambda_*)+N \label{nip}
\end{eqnarray}
with $N$ the nilpotent matrix, and
\begin{equation}\label{jordanform}
B(\lambda_*)=ZJ(\lambda_*)Z^{-1}=Z
\diag(J_{\mu}(\lambda_*),\widetilde{J}(\lambda_*))Z^{-1}
\in \mathbb{C}^{m\times m}
\end{equation}
with $\widetilde{J}(\lambda_*)$ consisting of the Jordan blocks
corresponding to the eigenvalues of
$W^HAW$ that are not equal to $\mu$. Then the determinant
$\det(\widetilde{J}(\lambda_*))\not=0$ is independent of $\mu$.
For $\mu\not=\lambda_*$, $B(\lambda_*)$ is nonsingular, so is $J(\lambda_*)$.

Exploiting $|\det(Z)\det(Z^{-1})|=1$,
by the basic property that the absolute value of a matrix determinant
is equal to that of the product of singular values or eigenvalues of the
underlying matrix, we obtain
\begin{eqnarray}
|\det(B(\lambda_*))|&=& |\det(\widetilde{J}(\lambda_*))||\det(J_{\mu}(\lambda_*))|\nonumber\\
&=&|\det(\widetilde{J}(\lambda_*))||\mu-\lambda_*|^{m_{\mu}}\label{deteig}\\
&=&|\det(\widetilde{J}(\lambda_*))|\left(\prod_{i=2}^{m_{\mu}}
\sigma_i(J_{\mu}(\lambda_*))\right)\sigma_{\min}(J_{\mu}(\lambda_*)) \nonumber\\
&=:&|\det(\widetilde{J}(\lambda_*))|c(\lambda_*)\sigma_{\min}(J_{\mu}(\lambda_*)),
\label{det}
\end{eqnarray}
where
$$
\sigma_{\min}(J_{\mu}(\lambda_*))\leq\sigma_2(J_{\mu}(\lambda_*))
\leq\sigma_3(J_{\mu}(\lambda_*))\leq\cdots\leq \sigma_{m_{\mu}}(J_{\mu}(\lambda_*))
$$
are the $m_{\mu}$ singular values of $J_{\mu}(\lambda_*)$. Therefore, from
the non-singularity of $B(\lambda_*)$, \eqref{deteig} and \eqref{det}, we obtain
\begin{equation} \label{jordanerror}
\sigma_{\min}(J_{\mu}(\lambda_*))=\frac{|\mu-\lambda_*|^{m_{\mu}}}{c(\lambda_*)}.
\end{equation}

From \eqref{jordan} and \eqref{nip}, we
regard $J_{\mu}(\lambda_*)$ as a perturbed $N$ with the perturbation matrix
$\diag(\mu-\lambda_*)$.
Since $N^TN$ is the identity matrix with the first diagonal entry replaced by
zero, the singular values of $N$ consist of simple 0 and 1
with multiplicity $m_{\mu}-1$. Therefore,
by a result due to Weyl on singular values \cite[Corollary 4.31, p.70]{Stewart1998}, we have
$$
|\sigma_i(J_{\mu}(\lambda_*))-1|\leq |\mu-\lambda_*|,\ i=2,3,\ldots,m_{\mu},
$$
i.e.,
\begin{equation}\label{svalue2}
 1-|\mu-\lambda_*|\leq \sigma_i(J_{\mu}(\lambda_*))\leq 1+|\mu-\lambda_*|,\ i=2,3,\ldots,m_{\mu},
\end{equation}
Hence, once $\varepsilon$ is small so that $|\mu-\lambda_*|<1$, the above lower bound is positive.
In this case, we obtain
$$
(1-|\lambda_*-\mu|)^{m_{\mu}-1}\leq c(\lambda_*)\leq (1+|\lambda_*-\mu|)^{m_{\mu}-1},
$$
which mean that $c(\lambda_*)$ is uniformly bounded above and below by
$1+\delta$ and $1-\delta$ with $\delta>0$ arbitrarily small, respectively, as
$\varepsilon\rightarrow 0$. By \eqref{jordanerrorlu} and the above, we obtain \begin{equation}\label{jordanerrorlu}
\frac{|\mu-\lambda_*|^{m_{\mu}}}{(1+|\mu-\lambda_*|)^{m_{\mu}-1}}  \leq \sigma_{\min}(J_{\mu}(\lambda_*))\leq\frac{|\mu-\lambda_*|^{m_{\mu}}}{(1-|\mu-\lambda_*|)^{m_{\mu}-1}}.
\end{equation}

From \eqref{jordanform}, it is easily justified that
$$
\frac{1}{\|Z\|\|Z^{-1}\|}\sigma_{\min}(J_{\mu}(\lambda_*))
\leq\sigma_{\min}(B(\lambda_*))
\leq\|Z\|\|Z^{-1}\|\sigma_{\min}(J_{\mu}(\lambda_*)).
$$
Therefore, it follows from \eqref{jordanerror} that
\begin{equation}\label{BJmin}
\sigma_{\min}(B(\lambda_*))=O(\sigma_{\min}(J_{\mu}(\lambda_*))=O(|\mu-\lambda_*|^{m_{\mu}}).
\end{equation}
In terms of \eqref{BJmin} and \eqref{sigmB}, we obtain \eqref{ritzerror}.
\end{proof}


\begin{remark}
Theorem 8 in Kahan, Parlett and Jiang \cite{kahan1982} presents a well-known perturbation
bound on the eigenvalues of a general non-normal matrix. A careful investigation
shows that their result and its proof can be adapted to the current context but do not
suffice to establish the error estimate \eqref{ritzerror}. In the current
context, Theorem 8 of \cite{kahan1982} states
\begin{equation}\label{kahanbound}
\frac{|\mu-\lambda_*|^{m_{\mu}}}{(1+|\mu-\lambda_*|)^{m_{\mu}-1}}
\leq \|ZE(\lambda_*)Z^{-1}\|
\end{equation}
with $E(\lambda_*)$ being the error matrix defined in Theorem~\ref{l1}. Imitating the proof
of the Bauer--Fike theorem \cite[p.357]{GoVa13}, they obtain bound~\eqref{kahanbound}
by proving that
\begin{equation}\label{estsvaluemin}
\frac{|\mu-\lambda_*|^{m_{\mu}}}{(1+|\mu-\lambda_*|)^{m_{\mu}-1}}\leq\sigma_{\min}(J_{\mu}(\lambda_*))
\leq \|ZE(\lambda_*)Z^{-1}\|,
\end{equation}
whose lower bound for $\sigma_{\min}(J_{\mu}(\lambda_*))$ is exactly the
one in \eqref{jordanerrorlu}. Here we point out that the bound $\sigma_{\min}(J_{\mu}(\lambda_*))
\leq \|ZE(\lambda_*)Z^{-1}\|$ cannot be improved, as is seen from the proof of Theorem 8 in
\cite{kahan1982}.
\end{remark}

\begin{remark}
Notice that $\sigma_{\min}(J_{\mu}(\lambda_*))\leq
\|ZE(\lambda_*)Z^{-1}\|=O(\varepsilon)$. Bound \eqref{estsvaluemin} indicates that
$\mu$ converges to $\lambda_*$ {\em at least} as fast as and may be {\em faster} than the error estimate
$O(\varepsilon^{1/m_{\mu}})$ in \eqref{ritzerror}. Rigorously speaking, bound
\eqref{kahanbound} alone may not lead to the error estimate
\eqref{ritzerror} unless it is tight in the sense that the ratio of
$\frac{|\mu-\lambda_*|^{m_{\mu}}}{(1+|\mu-\lambda_*|)^{m_{\mu}-1}}$
over $\sigma_{\min}(J_{\mu}(\lambda_*))$ is {\em uniformly} bounded by some {\em positive}
constant as $\mu\rightarrow\lambda_*$. If this ratio tends to {\em zero}
as $\mu\rightarrow\lambda_*$, then estimate \eqref{ritzerror} does {\em not} hold and
the convergence order of $\mu$ is {\em bigger} than $1/m_{\mu}$; that is,
$\mu$ will converge to $\lambda_*$ {\em faster} than
$O(\varepsilon^{1/m_{\mu}})$ , i.e., it will holds that $|\mu-\lambda_*|=o(\varepsilon^{1/m_{\mu}})$.
We notice that the proof approach itself of Theorem 8 in
\cite{kahan1982}  tells us little on the tightness of the lower bound for $\sigma_{\min}(J_{\mu}(\lambda_*))$.
A key step of their proof is an application of Gershgorin's disk theorem \cite[p.357]{GoVa13} to the Hermitian
semi-positive definite matrix $J_{\mu}(\lambda_*)J_{\mu}(\lambda_*)^H$,
and bounds its eigenvalues, i.e., the squares of singular values of $J_{\mu}(\lambda_*)$,
from above by $(1+|\mu-\lambda_*|)^2$. The
tightness of bounds $(1+|\mu-\lambda_*|)^2$ for the eigenvalue estimates
is not clear and needs a further analysis.
Based on the resulting upper bounds $1+|\mu-\lambda_*|$
for the singular values of $J_{\mu}(\lambda_*)$ and the value of
determinant of $J_{\mu}(\lambda_*)J_{\mu}(\lambda_*)^H$, they establish the lower bound
in \eqref{estsvaluemin} for $\sigma_{\min}(J_{\mu}(\lambda_*))$. The above analysis indicates that the lower
bound in \eqref{estsvaluemin}
for $\sigma_{\min}(J_{\mu}(\lambda_*))$ alone does not suffice to precisely characterize
the convergence order of $\mu$.
\end{remark}

\begin{remark}
A determinstic approach to precisely determining the convergence
order of $\mu$ is that if one can find an upper bound for
$\sigma_{\min}(J_{\mu}(\lambda_*))$ and proves that it is the same order small as a given
lower bound then the precise convergence order of $\mu$
is exactly the convergence order of the lower and upper bounds.
Using a proof approach different from that used in
\cite{kahan1982}, we establish the lower and upper bounds \eqref{jordanerrorlu} for $\sigma_{\min}(J_{\mu}(\lambda_*))$. In our proof, by \eqref{jordan} and \eqref{nip}, we
regard $J_{\mu}(\lambda_*)$ as a perturbed
matrix of the nilpotent matrix $N$ with the perturbation ${\rm diag}(\mu-\lambda_*)$,
and then exploit Weyl's theorem to have established both the lower and upper bounds in \eqref{svalue2}
for all the singular values but the smallest $\sigma_{\min}(J_{\mu}(\lambda_*))$
of $J_{\mu}(\lambda_*)$.
Relation \eqref{jordanerrorlu} shows that
the lower and upper bounds for $\sigma_{\min}(J_{\mu}(\lambda_*))$ are (asymptotically) the tightest since they
both tend to $|\mu-\lambda_*|^{m_{\mu}}$ as $\varepsilon\rightarrow 0$, i.e.,
$$
\frac{|\mu-\lambda_*|^{m_{\mu}}}{\sigma_{\min}(J_{\mu}(\lambda_*))}\rightarrow 1,
$$
indicating that
$$
|\mu-\lambda_*|\rightarrow (\sigma_{\min}(J_{\mu}(\lambda_*)))^{1/m_{\mu}}.
$$
The remaining proof of Theorem~\ref{linearconv} is
based on the obtained
lower and upper bounds for $\sigma_{\min}(J_{\mu}(\lambda_*))$ and ultimately achieves the
error estimate \eqref{ritzerror}. In summary,
as far as estimating $\sigma_{\min}(J_{\mu}(\lambda_*))$ is concerned,
it clear that \eqref{estsvaluemin} is a {\em part} of \eqref{jordanerrorlu} and, in the meantime,
Theorem~\ref{linearconv} is not a corollary of Theorem 8 in \cite{kahan1982}.
In addition, there is a distinction between Theorem 8 in \cite{kahan1982} and Theorem~\ref{linearconv}:
the former has no limitation on the size of $\|E(\lambda_*)\|$ in
\eqref{kahanbound}, but the latter requires that $\|E(\lambda_*)\|$, i.e.,
$\varepsilon$, be sufficiently small so that the lower bound in \eqref{svalue2} is positive, i.e.,
$|\mu-\lambda_*|<1$.
\end{remark}

Now let us deviate from the topic of NEP for a moment and
consider the eigenvalue perturbation problem of $\tilde A=A+E$ with
$A$ being a general nonnormal matrix and $E$ being a perturbation matrix.
Adapted the proof of Theorem \ref{linearconv} to it, it
is straightforward to establish
the following error estimate, which cannot be obtained from Theorem 8 of
\cite{kahan1982} as has been argued above.

\begin{corollary}
Let $\tilde{A}=A+E$ with $E$ a perturbation matrix and $J=ZAZ^{-1}$
be the Jordan normal form of $A$. For $\|E\|$ sufficiently small, to any
eigenvalue $\tilde{\lambda}$ of $\tilde{A}$ there corresponds an eigenvalue $\lambda$ of
$A$ such that if $\tilde{\lambda}\not=\lambda$ then
\begin{equation}
|\lambda-\tilde{\lambda}|=O(\|E\|^{\frac{1}{m_{\lambda}}}),
\end{equation}
where $m_{\lambda}$ is the order of the largest Jordan block to which $\lambda$ belong.
\end{corollary}

We can extend Theorem~\ref{linearconv} to the Rayleigh--Ritz method for NEP \eqref{NEP}.
To this end, we need the following famous Keldysh theorem \cite{Guttel}.

\begin{lemma}\label{Keldysh}
Suppose that the analytic matrix-valued function $B(\lambda)\in
\mathbb{C}^{m\times m}$ with $\lambda\in\Omega\subseteq \mathbb{C}$ has
finitely many eigenvalues $\mu_i,\ i=1,\ldots,s$ in $\Omega$ with
partial multiplicities $m_{i,1}\geq m_{i,2}\geq \cdots\geq m_{i,d_i}$,
and define
\begin{equation}\label{hatm}
 \widehat{m}=\sum_{i=1}^s\sum_{j=1}^{d_i}m_{i,j}.
\end{equation}
Then there are $m\times\widehat{m}$ matrices $Z$ and $Q$ whose columns
are generalized eigenvectors of $B(\lambda)$ and an
$\widehat{m}\times\widehat{m}$ Jordan
matrix $J$ whose eigenvalues are $\mu_i,\ i=1,\ldots,s$ with multiplicities $m_{i,1},\ldots,m_{i,d_i}$
such that
\begin{equation}\label{jordanB}
B(\lambda)^{-1}=Z(\lambda I-J)^{-1}Q^H+\widetilde{R}(\lambda)
\end{equation}
for some $m\times m$ analytic matrix-valued function $\widetilde{R}(\lambda)$
in $\Omega$ with
\begin{eqnarray}
J=\diag(J_1,J_2,\ldots,J_s), \ \ & & J_i=\diag(J_{i,1},\ldots,J_{i,d_i}),\label{defJ} \\
Z=(Z_1,\ldots,Z_s),\ \  && Z_i=(Z_{i,1},\ldots,Z_{i,d_i}),\label{defZ}\\
Q=(Q_1,\ldots,Q_s),\ \ && Q_i= (Q_{i,1},\ldots,Q_{i,d_i}), \label{defQ}
\end{eqnarray}
where $J_{i,j}\in \mathbb{C}^{m_{ij}\times m_{ij}}$ and $Z_i,Q_i$ are conformally partitioned as $J_i$,
$\ i=1,2,\ldots,s, j=1,2,\ldots,d_i$.
\end{lemma}

\begin{remark}
This theorem shows that (i) up to an analytic additive term $\widetilde{R}(\lambda)$,
the behavior of the resolvent $B(\lambda)^{-1}$ at a small neighborhood of each of the
eigenvalues $\mu_i$ is captured by the inverse of the shifted Jordan matrix $(\lambda I-J)^{-1}$,
(ii) an eigenvalue $\mu_i$ is called semi-simple if $m_{i,1}=\cdots=m_{i,d_i}=1$,
and it is simple if $d_i=1$, and (iii) the algebraic and
geometric multiplicities of $\mu_i$ are $\sum_{j=1}^{d_i}m_{i,j}$ and
$d_i$, respectively; (iv) $\widetilde{R}(\lambda)=0$ in the linear
case $B(\lambda)=W^HAW-\lambda I$.
\end{remark}


\begin{theorem} \label{ritznep}
Suppose that $B(\lambda)$ has finitely many eigenvalues
$\mu_i,\ i=1,\ldots,s$ in $\Omega\subseteq \mathbb{C}$ with
partial multiplicities $m_{i,1}\geq m_{i,2}\geq \cdots\geq m_{i,d_i}$
and the inverse of its resolvent is of form \eqref{jordanB}, and assume that
$\widehat{m}\leq m$ with $\widehat{m}$ defined by
\eqref{hatm} and $Z$, $Q$ are of full column rank.
For $\mu$ defined in \eqref{ritzvalue}, write $\mu:=\mu_1$. Suppose $\mu\not=\lambda_*$.
Then for $\mu$ lying in the closed disc $\mathcal{D}=\{\mu\in\Omega: \
|\mu-\lambda_{*}|\leq r\}\subset \Omega$, the following error estimate holds:
\begin{equation}\label{ritzbound}
 |\mu-\lambda_*|=O(\varepsilon^{1/m_{\mu}}),
\end{equation}
where $m_{\mu}:=m_{1,1}$ is the order of the largest Jordan block $J_{1,1}$ in $J_1$
defined by \eqref{defJ}.
\end{theorem}

\begin{proof}
Since $Z$ and $Q$ defined by \eqref{defZ} and \eqref{defQ} are of full column rank,
$(\lambda_* I-J)^{-1}Q^H$ is of full row rank. Therefore,
\begin{equation}\label{moore}
(Z(\lambda_*I-J)^{-1}Q^H)^{-1}=(Q^H)^{\dagger}(\lambda_*I-J)Z^{\dagger},
\end{equation}
where $\dagger$ denotes the Moore--Penrose generalized
inverse of a matrix. Denote by $\sigma_{\max}(\cdot)$ the
largest singular value of a matrix. Since $\widetilde{R}(\lambda)$
is analytic in $\lambda\in \Omega$ and $\lambda_*\in\mathcal{D}\subset
\Omega$,
$\|\widetilde{R}(\lambda)\|$ is uniformly bounded with respect to $\lambda\in\mathcal{D}\subset \Omega$.
Notice that
$$
 \sigma_{\max}(B(\lambda_*)^{-1})\mbox{\ \ and\ \ }\sigma_{\max}(Z(\lambda_*
I-J)^{-1}Q^H)
$$
are infinitely large as $\mu\rightarrow \lambda_*$, since, by \eqref{sigmB} and \eqref{moore},
their reciprocals
$$
\sigma_{\min}(B(\lambda_*))\mbox{\ \ and\ \ }
\sigma_{\min}((Q^H)^{\dagger}(\lambda_*I-J)Z^{\dagger}).
$$
tend to zero as $\mu\rightarrow \lambda_*$, whose convergence has been known to be unconditional
as $\varepsilon\rightarrow 0$ (cf. the paragraph after Theorem~\ref{C1}).

From \eqref{jordanB}, we have
{\small
$$
\sigma_{\max}(Z(\lambda_*
I-J)^{-1}Q^H)-\|\widetilde{R}(\lambda_*)\|\leq \sigma_{\max}(B(\lambda_*)^{-1})\leq\sigma_{\max}(Z(\lambda_*
I-J)^{-1}Q^H)+\|\widetilde{R}(\lambda_*)\|.
$$}
Therefore, as $\varepsilon\rightarrow 0$, we have
\begin{equation}\label{sim}
\sigma_{\max}(B(\lambda_*)^{-1})\sim \sigma_{\max}(Z(\lambda_*
I-J)^{-1}Q^H),
\end{equation}
where $\sim$ indicates that the ratio of the left and right hand sides
is asymptotically one.

In $\lambda I-J$, define $J_{\mu}(\lambda)=\lambda I-J_1$, which corresponds to
the Ritz value $\mu=\mu_1$, and
$J_i(\lambda)=\lambda I-J_i,\ i=2,3,\ldots,s$. Then, as
$\varepsilon\rightarrow 0$, from \eqref{moore} and \eqref{sim} we obtain
\begin{eqnarray*}
 \sigma_{\min}(B(\lambda_*))&\sim &
\sigma_{\min}((Q^H)^{\dagger}(\lambda_*I-J)Z^{\dagger})=O(\sigma_{\min}(J_{\mu}(\lambda_*))).
\end{eqnarray*}
Since $m_{1,1}\geq\cdots\geq m_{1,d_1}$, exploiting \eqref{jordanerror}
and $m_{\mu}=m_{1,1}$, we obtain
$$
\sigma_{\min}(J_{\mu}(\lambda_*))=\sigma_{\min}(\lambda_* I-J_{1,1})=
\frac{|\lambda_*-\mu|^{m_{\mu}}}{c(\lambda_*)},
$$
which is just \eqref{jordanerror}.
The remaining proof is the same as that of Theorem~\ref{linearconv} after \eqref{jordanerror}.
\end{proof}

\begin{remark}
It is worth noticing that $m_{\mu}$
depends on $\varepsilon$ and may change as $\varepsilon$ changes.
If $m_{\mu}=1$, as is often the case when $\lambda_*$ is a simple eigenvalue of
$T(\cdot)$, then the convergence of $\mu$
is linear and is as fast as $\varepsilon\rightarrow 0$. Otherwise, the convergence can be slow.
\end{remark}


\section{Convergence of the Ritz vector}\label{sec:R}

In this section, we will show that the Ritz vector $\widetilde x$ converges conditionally,
meaning that the Rayleigh--Ritz method may
fail to compute $x_*$.
We will show that $\widetilde x$ even may not be unique,
which corresponds to the case that the geometric multiplicity
of $\mu$ is bigger than one, even if the desired $\lambda_*$ is
simple. In this case, there are more than one Ritz vectors $\widetilde x$,
each of which is generally a meaningless approximation to
$x_*$ even if $\varepsilon=0$, causing that the method fails
even if we have a perfect projection subspace $\mathcal{W}$ that contains
$x_*$ exactly. Furthermore, even if $\widetilde x$ is unique, we will prove that there is no guarantee
that it converges to $x_*$ as $\varepsilon\rightarrow 0$.
Later on, we will construct examples to illustrate these assertions.

Let $X_{\bot}$ be an orthonormal basis of the orthogonal complement of
${\rm span}\{x_{*}\}$ in $\mathbb{C}^n$. Then $(x_*\ X_{\perp})$ is unitary.
From $T(\lambda_{*})x_{*}=0$, we obtain the Schur-like decomposition of $T(\lambda_*)$:
\begin{equation}\label{schurlike}
	\begin{pmatrix}
		x_{*}^{H} \\
		X_{\bot}^{H} \\
	\end{pmatrix}T(\lambda_{*})
(x_{*} \ X_{\bot})=\begin{pmatrix}
		0 & x_{*}^{H}T(\lambda_{*})X_{\bot}\\
		0 &  L(\lambda_{*}) \\
	\end{pmatrix},
\end{equation}
where $L(\lambda_{*})=X_{\bot}^{H}T(\lambda_{*})X_{\bot}$.
For a simple $\lambda_{*}$, Proposition 1 in \cite{Neumaier} states
that $\text{rank}(T(\lambda_{*}))=n-1$. As a result,
$L(\lambda_{*})$ is nonsingular, and
$\sigma_{\min}(L(\lambda_{*}))>0$.

Notice that, for a given (not necessarily Ritz)
approximation $(\mu,\widetilde{x})$ to $(\lambda_*,x_*)$, its residual norm
$\|T(\mu)\widetilde{x}\|$ is a posteriori computable.
In terms of $\|T(\mu)\widetilde{x}\|$, the following theorem establishes
an error bound for $\sin\angle(\widetilde{x},x_*)$, which extends Theorem 3.1
of \cite{Jia1999} and Theorem 6.1 of \cite{Jia2001} to the NEP case
and applies to the Rayleigh--Ritz method
and the refined Rayleigh--Ritz method described in section 1.

\begin{theorem}\label{con-eigenpair}
Let $(\mu,\widetilde{x})$ with $\|\widetilde{x}\|=1$ be an arbitrarily given approximation
to the desired eigenpair $(\lambda_*,x_*)$ of $T(\lambda)x=0$, denote the
corresponding residual norm by
$$\rho=\| T(\mu)\widetilde{x}\|,
$$
and write $L(\mu)=X_{\bot}^{H}T(\mu)X_{\bot}$. If $\mu\in\mathcal{D}$ with the closed
disc $\mathcal{D}$ defined in Theorem~{\rm\ref{ritznep}} and
$\sigma_{\min}(L(\mu))>0$, then
\begin{equation}\label{sin-x}	
\sin\angle(\widetilde{x},x_*)\leq\frac{\rho+
\|T^{\prime}(\lambda_{*})\||\mu-\lambda_{*}|}{\sigma_{\min}(L(\mu))}+O(|\mu-\lambda_*|^2).
\end{equation}
\end{theorem}

\begin{proof}
Exploiting $x_{*}x_{*}^{H}+X_{\bot}X_{\bot}^{H}=I$,
by the 2-norm unitary invariance, we obtain
\begin{eqnarray*}
	\rho &=& \bigg\|\begin{pmatrix}
		x_{*}^{H} \\
		X_{\bot}^{H} \\
	\end{pmatrix}T(\mu)\widetilde{x}\bigg\| \\
	&=&  \bigg\|\begin{pmatrix}
		x_{*}^{H}T(\mu)\widetilde{x} \\
		X_{\bot}^{H}T(\mu)(x_{*}x_{*}^{H}+X_{\bot}X_{\bot}^{H})\widetilde{x}\\
	\end{pmatrix}\bigg\|
	\\
	&=&  \bigg\|\begin{pmatrix}
		x_{*}^{H}T(\mu)\widetilde{x} \\
		X_{\bot}^{H}T(\mu)x_{*}x_{*}^{H}\widetilde{x}+L(\mu)X_{\bot}^{H}\widetilde{x} \\
	\end{pmatrix} \bigg\|.
\end{eqnarray*}
Therefore,
\begin{eqnarray*}
	\rho&\geq& \|X_{\bot}^{H}T(\mu)x_{*}x_{*}^{H}\widetilde{x}+L(\mu)X_{\bot}^{H}\widetilde{x}\|\\
	&\geq& \|L(\mu)X_{\bot}^{H}\widetilde{x}\|-\|X_{\bot}^{H}T(\mu)x_{*}x_{*}^{H}\widetilde{x}\|,
\end{eqnarray*}
i.e.,
\begin{equation}\label{p}
	\|L(\mu)X_{\bot}^{H}\widetilde{x}\|\leq\rho+
	\|X_{\bot}^{H}T(\mu)x_{*}x_{*}^{H}\widetilde{x}\|.
\end{equation}
Since $T(\mu)$ is analytic in $\mu\in\mathcal{D}\subset\Omega$, we have
\begin{equation}\label{expension}
	T(\mu)=T(\lambda_{*})+T^{\prime}(\lambda_{*})(\mu-\lambda_{*})+O((\mu-\lambda_{*})^2),
\end{equation}
which, by making use of $T(\lambda_{*})x_*=0$, proves
$$
T(\mu)x_{*}=(\mu-\lambda_{*})T^{\prime}(\lambda_{*})x_{*}+O((\mu-\lambda_{*})^2).
$$
As a result,
\begin{equation}\label{inest}
	\|X_{\bot}^{H}T(\mu)x_{*}x_{*}^{H}\widetilde{x}\|\leq \|T^{\prime}(\lambda_{*})\||\mu-\lambda_{*}|
+O(|\mu-\lambda_*|^2).
\end{equation}

Notice that
$$
\sin\angle(x_{*},\widetilde{x})=\|X_{\bot}^H\widetilde{x}\|
\mbox{\ \ and\ \ }\|L(\mu)X_{\bot}^{H}\widetilde{x}\|\geq \sigma_{\min}(L(\mu))\|X_{\bot}^{H}\widetilde{x}\|.
$$
Therefore, from (\ref{p}) and (\ref{inest}), we obtain
$$\sin\angle(x_{*},\widetilde{x})\sigma_{\min}(L(\mu))
\leq\rho+\|T^{\prime}(\lambda_{*})\||\mu-\lambda_{*}|+O(|\mu-\lambda_*|^2),
$$
proving that
\begin{equation*}
	\sin\angle(x_{*},\widetilde{x})\leq\dfrac{\rho+\|T^{\prime}(\lambda_{*})\|
|\mu-\lambda_{*}|} {\sigma_{\min}(L(\mu))}+O(|\mu-\lambda_*|^2).
\end{equation*}
~
\end{proof}

\begin{remark}\label{remres}
The second term in the right-hand side of \eqref{p} vanishes
for the linear eigenvalue problem $T(\lambda)x=(A-\lambda I)x=0$ since
$X_{\bot}^HT(\mu)x_*=(\lambda_*-\mu)X_{\bot}^Hx_*=0$. In this case,
the right-hand side of \eqref{sin-x} reduces to $\rho/\sigma_{\min}(L(\mu))$, and
Theorem~\ref{con-eigenpair} reduces to Theorem 3.1 of \cite{Jia1999}
for the linear eigenvalue problem.
\end{remark}

For the simple $\lambda_*$, notice that $\sigma_{\min}(L(\mu))\rightarrow
\sigma_{\min}(L(\lambda_*))>0$ as $\mu\rightarrow\lambda_*$. Therefore,
the condition $\sigma_{\min}(L(\mu))>0$ must be fulfilled, provided
that $|\mu-\lambda_*|$ is sufficiently small. We will investigate more on
$\sigma_{\min}(L(\mu))$. Theorem~\ref{con-eigenpair} indicates that if
$\mu\rightarrow\lambda_*$ and the residual norm $\rho$ converges to zero
then the corresponding approximate eigenvector $\widetilde{x}\rightarrow x_*$.
Therefore, $\rho$ can be used to check if the approximate eigenpair $(\mu,\widetilde{x})$ converges.

Now let us return to the convergence problem of the Ritz vector $\widetilde{x}$
as $\varepsilon\rightarrow 0$. We will prove that it
converges to $x_*$ {\em conditionally} and reveal why it may fail to converge. Particularly,
we will illustrate that the method may not find $x_*$ and $\widetilde{x}$ may
not be unique even when $\varepsilon=0$, i.e., $x_*\in \mathcal{W}$.

Recall from Algorithm~\ref{alg-RR} that $B(\mu)z=0$. Let $(z \ \ Z_{\bot})$ be unitary. Then
we obtain a Schur-like decomposition of $B(\mu)$:
\begin{equation}\label{schurB}
	\begin{pmatrix}
		z^{H} \\
		Z_{\bot}^{H} \\
	\end{pmatrix}B(\mu)
		(z\ \ Z_{\bot}) \\
	=\begin{pmatrix}
		0 & z^{H}B(\mu)Z_{\bot}\\
		0 &  C(\mu) \\
	\end{pmatrix},
\end{equation}
where $C(\mu)=Z_{\bot}^HB(\mu)Z_{\bot}$.
If $\mu$ is a simple eigenvalue of $B(\cdot)$, then $C(\mu)$ is nonsingular.
However, if $\mu$ is multiple or, though simple,
is close to some other eigenvalues
of $B(\cdot)$, then $C(\mu)$ is singular or is close to singularity. In
the first case, since $\sigma_{\min}(C(\mu))$ is zero, the
eigenvector $z$ and thus the Ritz vector $\widetilde{x}$
are {\em not} unique, leading to the {\em failure} of the method.
In the second case, $\sigma_{\min}(C(\mu))$ is close to zero; we will
prove that $\widetilde{x}$, though unique, may be
a very inaccurate approximation to $x_*$ and even may have no accuracy.

Suppose that $\mu$ is simple. Then the Ritz vector $\widetilde{x}$ is unique.
The following result gives an a priori error bound for the error of
$\widetilde{x}$ in terms of $\varepsilon$, and shows how the size of
$\sigma_{\min}(C(\lambda_*))$ affects the accuracy of $\widetilde{x}$.
Keep in mind the basic fact that if $\mu$ is close to some
other eigenvalues of $B(\cdot)$ then $\sigma_{\min}(C(\mu))$ is close to zero and; in this case,
by the continuity argument, $\sigma_{\min}(C(\lambda_*))$ is close to
$\sigma_{\min}(C(\mu))$ and is thus close to zero.

\begin{theorem}\label{con-Ritz}
Assume that $\mu$ is simple and $\sigma_{\min}(C(\lambda_*))>0$.
Then $\widetilde{x}$ is unique and 	
\begin{equation}\label{sin-x1}
		\sin\angle(\widetilde{x},x_*)\leq\bigg(1+\frac{\|
T(\lambda_{*})\|}{\sqrt{1-\varepsilon^{2}}\sigma_{\min}(C(\lambda_*))}\bigg)
		\varepsilon+
\frac{\|T^{\prime}(\lambda_{*})\||\mu-\lambda_{*}|}{\sigma_{\min}(C(\lambda_*))}
+O(|\mu-\lambda_*|^2).	
\end{equation}
\end{theorem}

\begin{proof}
Note that the residual
$r=B(\lambda_{*})\widehat{u}$ in (\ref{err}) satisfies
$$\| r\|\leq\frac{\varepsilon}{\sqrt{1-\varepsilon^{2}}}\| T(\lambda_{*})\|.$$
Then applying Theorem \ref{con-eigenpair} to $B(\mu)z=0$ and the residual
norm $\|r\|=\|B(\lambda_{*})\widehat{u}\|$ with $(\lambda_*,\widehat{u})$ being an
approximate eigenpair of $B(\cdot)$, we obtain
\begin{equation}\label{sin-int}
\sin\angle(z,\widehat{u})\leq\frac{\| T(\lambda_{*})\|\varepsilon
+\sqrt{1-\varepsilon^{2}}\|T^{\prime}(\lambda_{*})\||\mu-\lambda_{*}|}
{\sqrt{1-\varepsilon^{2}}\sigma_{\min}(C(\lambda_*))}+O(|\mu-\lambda_*|^2).
\end{equation}
Since $\widetilde{x}=Wz$, $\widehat{u}=W^Hx_*/\sqrt{1-\varepsilon^2}$
and $P_{\mathcal{W}}=WW^H$, \eqref{sin-int} means that
$$
\sin\angle(\widetilde{x},P_{\mathcal{W}}x_{*})\leq
\frac{\| T(\lambda_{*})\|\varepsilon
+\sqrt{1-\varepsilon^{2}}\|T^{\prime}(\lambda_{*})\||\mu-\lambda_{*}|}
{\sqrt{1-\varepsilon^{2}}\sigma_{\min}(C(\lambda_*))}+O(|\mu-\lambda_*|^2)
$$
due to the orthonormality of $W$. Exploiting the triangle inequality
$$
\angle(\widetilde{x},x_{*})\leq\angle(x_{*},P_{\mathcal{W}}x_{*})
+\angle(P_{\mathcal{W}}x_{*},\widetilde{x})
=\angle(x_{*},\mathcal{W})+\angle(P_{\mathcal{W}}x_{*},\widetilde{x}),
$$
by \eqref{devia} we obtain
\begin{eqnarray*}
	\sin\angle(\widetilde{x},x_{*})&\leq&
	\sin\angle(x_{*},\mathcal{W})+\sin\angle(P_{\mathcal{W}}x_{*},\widetilde{x})\\
	&\leq& \bigg(1+\dfrac{\| T(\lambda_{*})\|}{\sqrt{1-\varepsilon^{2}}\sigma_{\min}(C(\lambda_*))}\bigg)
	\varepsilon+\frac{\|T^{\prime}(\lambda_{*})\||\mu-\lambda_{*}|}
{\sigma_{\min}(C(\lambda_*))}+O(|\mu-\lambda_*|^2).
\end{eqnarray*}
~
\end{proof}

Since \eqref{sin-int} is an application of Theorem~\ref{con-eigenpair}
to $r=B(\lambda_*)\widehat{u}$, Remark~\ref{remres} works for \eqref{sin-int},
the second and third terms in the right-hand side
of \eqref{sin-x1} vanish when NEP~\eqref{NEP} becomes a linear one,
and Theorem~\ref{con-Ritz} reduces to Theorem 3.2 of \cite{Jia1999}.
Theorem~\ref{con-Ritz} provides sufficient conditions for the convergence of
the Ritz vector $\widetilde{x}$, which require that
$\sigma_{\min}(C(\lambda_*))$
be {\em uniformly bounded away from zero} as $\varepsilon\rightarrow 0$. Unfortunately,
as we have commented before the theorem, theoretically there is no guarantee
that this condition is satisfied. Therefore, $\widetilde{x}$ converges to $x_*$ conditionally.
If $\sigma_{\min}(C(\lambda_*))$ is small or
no larger than $O(\varepsilon)$, then $\widetilde{x}$ may be poor and even
may have no accuracy; in other words, $\widetilde{x}$ may fail to converge to $x_*$.

\section{Convergence of the refined Ritz vector} \label{sec:RR}

We consider the convergence of the refined Ritz vector $\widehat{x}$,
and prove that $\widehat{x}\rightarrow x_*$ unconditionally as
$\varepsilon\rightarrow 0$.

Let us replace $T(\lambda_*)$ in decomposition \eqref{schurlike} by $T(\mu)$.
Then, by the continuity argument, the entries of $(1,1)$ and $(2,1)$-positions
in the resulting decomposition tend to zero, and $L(\mu)\rightarrow L(\lambda_*)$ in $(2,2)$-position
as $\varepsilon\rightarrow 0$. Since $T(\lambda)$ is analytic in $\lambda\in \Omega$, the
matrix $L(\lambda)$ is analytic in $\lambda\in\Omega$ too. Therefore,
for $\mu\in\mathcal{D}\subset \Omega$ where the disc $\mathcal{D}$ is
defined as in Theorem~\ref{ritznep}, we obtain
the convergent Taylor series of $L(\mu)$ at $\lambda_*$ for $\mu\in\mathcal{D}$:
\begin{eqnarray}
L(\mu)&=&L(\lambda_{*})+L^{\prime}(\lambda_{*})(\mu-\lambda_{*})+\frac{L^{\prime\prime}
(\lambda_{*})}{2}(\mu-\lambda_{*})^{2}+O((\mu-\lambda_*)^3)\nonumber\\
&=&L(\lambda_{*})+L^{\prime}(\lambda_{*})(\mu-\lambda_{*})
+\left(\frac{L^{\prime\prime}(\lambda_{*})}{2}+O(\mu-\lambda_*)\right)(\mu-\lambda_{*})^{2}. \label{sig-0}
\end{eqnarray}
As a result, for $\mu\in\mathcal{D}$, since $|\mu-\lambda_*|\leq r$, we have
\begin{eqnarray}
\left\|\frac{L^{\prime\prime}(\lambda_{*})}{2}+O(\mu-\lambda_*)\right\|&\leq&
\frac{\|L^{\prime\prime}(\lambda_{*})\|}{2}+O(|\mu-\lambda_*|)\nonumber\\
&\leq& \frac{\|L^{\prime\prime}(\lambda_{*})\|}{2}+O(r)\nonumber\\
&\leq& \frac{\|L^{\prime\prime}(\lambda_{*})\|}{2}+\eta=:\beta, \label{beta}
\end{eqnarray}
which is uniformly bounded with respect to $\mu\in\mathcal{D}\subset\Omega$, provided that, for an
arbitrarily given constant $\eta$, the disc radius $r$ is sufficiently small such that $O(r)\leq\eta$.

\begin{theorem}\label{con-Refined-vector}
Let $\beta$ be defined as in \eqref{beta}, and assume that $|\mu-\lambda_{*}|$ is small enough and satisfies
	\begin{equation}\label{assumr}
		\sigma_{\min}(L(\lambda_{*}))-\| L^{\prime}(\lambda_{*})\|
|\mu-\lambda_{*}|-\beta |\mu-\lambda_{*}|^{2}>0.
	\end{equation}
Then
\begin{eqnarray}
\| T(\mu)\widehat{x}\|
&\leq& \frac{\| T(\mu)x_{*}\|+\| T(\mu)\|\varepsilon}{\sqrt{1-\varepsilon^{2}}} \label{rres1}\\
&\leq&\frac{\| T(\mu)\|\varepsilon+\|T^{\prime}(\lambda_*)\| |\mu-\lambda_{*}|
+\gamma|\mu-\lambda_*|^2}{\sqrt{1-\varepsilon^{2}}},
\label{rres}\\
\sin\angle(x_*,\widehat{x})
&\leq& \frac{\| T(\mu)x_{*}\|+\| T(\mu)\|\varepsilon}{\sqrt{1-\varepsilon^{2}}\sigma_{\min}(L(\mu))}
\label{errorrefine1}\\
&\leq&\frac{\| T(\mu)\|\varepsilon+\|T^{\prime}(\lambda_*)\||\mu-\lambda_*|
+\gamma|\mu-\lambda_*|^2}
	{\sqrt{1-\varepsilon^{2}}\sigma_{\min}(L(\mu))}
\label{errorrefine}
\end{eqnarray}
with $\sigma_{\min}(L(\mu))\geq \sigma_{\min}(L(\lambda_{*}))-\|
L^{\prime}(\lambda_{*})\| |\mu-\lambda_{*}|-\beta |\mu-\lambda_{*}|^{2}$
and the constant $\gamma$ defined in \eqref{gamma}.
\end{theorem}

\begin{proof}
Decompose $x_*$ as the orthogonal direct sum
$$
x_{*}=\frac{P_{\mathcal{W}}x_{*}}{\| P_{\mathcal{W}}x_{*}\|}\cos\angle(x_{*},\mathcal{W})
+\frac{(I-P_{\mathcal{W}})x_{*}}{\| (I-P_{\mathcal{W}})x_{*}\|}\sin\angle(x_{*},\mathcal{W}),
$$
and note that $\sin\angle(x_{*},\mathcal{W})=\varepsilon$ and
$\cos\angle(x_{*},\mathcal{W})=\sqrt{1-\varepsilon^2}$.
Then
\begin{eqnarray}
	\bigg\|T(\mu)\frac{P_{\mathcal{W}}x_{*}}{\| P_{\mathcal{W}}x_{*}\|}\bigg\| &=& \bigg\| \frac{T(\mu)}{\sqrt{1-\varepsilon^{2}}}\bigg(x_{*}-\frac{(I-P_{\mathcal{W}}x_{*})}{\| (I-P_{\mathcal{W}})x_{*}\|}\varepsilon\bigg)\bigg\| \nonumber \\
	&\leq& \frac{\| T(\mu)x_{*}\|+
\| T(\mu)\|\varepsilon}{\sqrt{1-\varepsilon^{2}}}. \label{resnorm2}
\end{eqnarray}
On the other hand, since $T(\lambda)$ is analytic in
$\lambda\in \Omega$, for $\mu\in\mathcal{D}\subset\Omega$
with the disc $\mathcal{D}$ defined in Theorem~\ref{ritznep},
we obtain
the convergent Taylor series of $T(\mu)$ at $\lambda_*$ for $\mu\in\mathcal{D}$:
\begin{eqnarray}
T(\mu)&=&L(\lambda_{*})+T^{\prime}(\lambda_{*})(\mu-\lambda_{*})+\frac{T^{\prime\prime}
(\lambda_{*})}{2}(\mu-\lambda_{*})^{2}+O((\mu-\lambda_*)^3)\nonumber\\
&=&T(\lambda_{*})+T^{\prime}(\lambda_{*})(\mu-\lambda_{*})
+\left(\frac{T^{\prime\prime}(\lambda_{*})}{2}+O(\mu-\lambda_*)\right)(\mu-\lambda_{*})^{2}. \label{sig-1}
\end{eqnarray}
As a result, for $\mu\in\mathcal{D}$, since $|\mu-\lambda_*|\leq r$, we have
\begin{eqnarray}
\left\|\frac{T^{\prime\prime}(\lambda_{*})}{2}+O(\mu-\lambda_*)\right\|&\leq&
\frac{\|T^{\prime\prime}(\lambda_{*})\|}{2}+O(|\mu-\lambda_*|)\nonumber\\
&\leq& \frac{\|T^{\prime\prime}(\lambda_{*})\|}{2}+O(r)\nonumber\\
&\leq& \frac{\|T^{\prime\prime}(\lambda_{*})\|}{2}+\eta=:\gamma, \label{gamma}
\end{eqnarray}
which is uniformly bounded in $\mu\in\mathcal{D}\subset\Omega$, provided that, for an
arbitrarily given constant $\eta$, the disc radius $r$ is sufficiently small such that $O(r)\leq\eta$.

From $T(\lambda_{*})x_{*}=0$, \eqref{sig-1} and \eqref{gamma}, we obtain
\begin{eqnarray*}
	\| T(\mu)x_{*}\| &=& \|( T(\mu)-T(\lambda_{*}))x_{*} \|\\
	&\leq& |\mu-\lambda_*|\|T^{\prime}(\lambda_*)\|+\gamma|\mu-\lambda_*|^2,
\end{eqnarray*}
Therefore, by definition (\ref{Refin}) of the refined Ritz vector
$\widehat{x}$ and bound \eqref{resnorm2}, we have
\begin{eqnarray*}
	\| T(\mu)\widehat{x}\|&\leq&
\bigg\| T(\mu)\frac{P_{\mathcal{W}}x_{*}}{\| P_{\mathcal{W}}x_{*}\|}\bigg\|\\
&\leq&\frac{\| T(\mu)x_{*}\|+\| T(\mu)\|\varepsilon}{\sqrt{1-\varepsilon^{2}}}\\
&\leq& \frac{\| T(\mu)\|\varepsilon+\|T^{\prime}(\lambda_*)\| |\mu-\lambda_{*}|+\gamma|\mu-\lambda_*|^2}{\sqrt{1-\varepsilon^{2}}},
\end{eqnarray*}
which proves \eqref{rres1} and \eqref{rres}.

From \eqref{sig-0} and \eqref{beta}, making use of
a result due to Weyl on singular values \cite[Corollary 4.31, p.70]{Stewart1998}, we have
$$
\sigma_{\min}(L(\mu))\geq\sigma_{\min}(L(\lambda_{*}))-\|
L^{\prime}(\lambda_{*})\|\mid\mu-\lambda_{*}\mid-\beta|\mu-\lambda_{*}|^{2}.
$$
As a result, under assumption \eqref{assumr}, by applying
Theorem \ref{con-eigenpair} to the residual norm
$\|T(\mu)\widehat{x}\|$ of the approximate eigenpair $(\mu,\widehat x)$,
it follows from \eqref{rres} that
\begin{eqnarray*}
\sin\angle(x_*,\widehat{x})&\leq&
\frac{\| T(\mu)x_{*}\|+\| T(\mu)\|\varepsilon}{\sqrt{1-\varepsilon^{2}}\sigma_{\min}(L(\mu))}
\\
&\leq& \frac{\| T(\mu)\|\varepsilon+\|T^{\prime}(\lambda_*)\| |\mu-\lambda_{*}|+\gamma|\mu-\lambda_*|^2}{\sqrt{1-\varepsilon^{2}}\sigma_{\min}
		(L(\mu))},
\end{eqnarray*}
which proves \eqref{errorrefine1} and \eqref{errorrefine}.
\end{proof}

Theorem \ref{con-Refined-vector} indicates that the
residual norm $\|T(\mu)\widehat{x}\|\rightarrow 0$ and $\widehat{x}\rightarrow
x_{*}$ as $\varepsilon\rightarrow 0$. Recall from Theorem~\ref{con-Ritz} that
the convergence of the Ritz vector $\widetilde{x}$ requires that
$\sigma_{\min}(C(\lambda_*))$ be uniformly bounded away from zero, which, however,
equals zero or can be arbitrarily small when $\mu$ is a multiple Ritz value
or is close to some other Ritz values. In contrast,
the situation is fundamentally different for the refined Ritz vector
$\widehat{x}$ since
$\sigma_{\min}(L(\mu))$ must be uniformly positive for a simple $\lambda_*$ as
$\sigma_{\min}(L(\mu))\rightarrow\sigma_{\min}(L(\lambda_*))>0$ as $\mu\rightarrow\lambda_*$.

In what follows we further compare the Ritz vector $\widetilde x$ with the refined Ritz
vector $\widehat x$, and get more insight into them.
For the simple $\lambda_*$, the following result shows that
$\sigma_{\min}(T(\mu))$ must be a simple singular value of $T(\mu)$ once $\mu$ is
sufficiently close to $\lambda_{*}$, which will be exploited to
establish the uniqueness of the refined Ritz vector $\widehat{x}$.


\begin{lemma}\label{unique-rR-value}
For the simple eigenvalue $\lambda_*$, denote by $\sigma_{2}(T(\lambda_{*}))$
the second smallest singular value of
$T(\lambda_{*})$, and let $\gamma$ be defined by \eqref{gamma}.
If $\mu$ is sufficiently close to $\lambda_{*}$ such that
	\begin{equation}\label{assums}
		\| T^{\prime}(\lambda_{*})\||\mu-\lambda_{*}|+\gamma|\mu-\lambda_{*}|^{2}
		<\dfrac{1}{2}\sigma_{2}(T(\lambda_{*})),
	\end{equation}
then $\sigma_{\min}(T(\mu))$ is simple.
\end{lemma}

\begin{proof}
Under the assumption that $\lambda_*$ is simple, $\sigma_{\min}(T(\lambda_{*}))=0$ is a simple
singular value of $T(\lambda_{*})$, which means that
$$\sigma_{2}(T(\lambda_{*}))>\sigma_{\min}(T(\lambda_{*}))=0.$$
From (\ref{sig-1}) and \eqref{gamma}, we obtain
$$
\| T(\mu)-T(\lambda_{*})\|\leq \| T^{\prime}(\lambda_{*})\|
|\mu-\lambda_{*}|+\gamma|\mu-\lambda_{*}|^{2}.
$$
Therefore, by the standard perturbation theory,
the errors of the $i$th smallest singular values $\sigma_i(\cdot)$ of $T(\mu)$
and those of $T(\lambda_*)$ satisfy
$$
|\sigma_{i}(T(\mu))-\sigma_{i}(T(\lambda_{*}))|\leq\|
T^{\prime}(\lambda_{*})\| |\mu-\lambda_{*}|+\gamma|\mu-\lambda_{*}|^{2}.
$$
Therefore,
\begin{equation}\label{rR1}
	\sigma_{2}(T(\mu))\geq\sigma_{2}(T(\lambda_{*}))-\| T^{\prime}(\lambda_{*})\|
|\mu-\lambda_{*}|-\gamma|\mu-\lambda_{*}|^{2}
\end{equation}
and, by \eqref{assums},
\begin{equation}\label{sigma1}
	\sigma_{\min}(T(\mu))=\sigma_1(T(\mu))\leq\| T^{\prime}(\lambda_{*})
\| |\mu-\lambda_{*}|+\gamma|\mu-\lambda_{*}|^{2}
	<\frac{1}{2}\sigma_{2}(T(\lambda_{*})),
\end{equation}
where \eqref{sigma1} holds because of
$\sigma_{\min}(T(\lambda_*))=\sigma_1(T(\lambda_*))=0$.

On the other hand, from (\ref{assums}) and (\ref{rR1}) we obtain
$$\sigma_{2}(T(\mu))>\frac{1}{2}\sigma_{2}(T(\lambda_{*})),$$
which, together with \eqref{sigma1}, shows that
$\sigma_{\min}(T(\mu))$ is a simple singular value of $T(\mu)$.
\end{proof}

Based on Lemma \ref{unique-rR-value}, we can prove that, unlike the Ritz
vector $\widetilde x$, the refined Ritz vector $\widehat x$ must be
unique as $\varepsilon\rightarrow 0$.

\begin{theorem}\label{unique-rR-vector}
For the simple $\lambda_*$, let $\widehat{\sigma}_1\leq\widehat{\sigma}_2\leq\cdots\leq \widehat{\sigma}_m$
be the singular values of $T(\mu)W$ with $m$ being the number of columns of $W$, and
$|\mu-\lambda_*|$ be sufficiently small such that
\begin{equation}  \label{assum2}
\widehat{\sigma}_{1}<\dfrac{1}{2}\sigma_{2}(T(\lambda_{*}))-
	\|T^{\prime}(\lambda_{*})\| |\mu-\lambda_{*}|
\mbox{\ \ and\ \ }
\sigma_2(T(\lambda_{*}))>2\gamma|\mu-\lambda_{*}|^{2}.
\end{equation}
Then $\widehat{\sigma}_{1}$ is simple, and $\widehat{x}$ is unique.
\end{theorem}

\begin{proof}
In terms of the eigenvalue interlacing property of the
Hermitian matrix
$$
\begin{pmatrix}
	W^{H} \\
	W_{\bot}^{H} \\
\end{pmatrix}T^{H}(\mu)T(\mu)
(W \ W_{\bot})=\begin{pmatrix}
	W^{H}T(\mu)^{H}T(\mu)W & W^{H}T(\mu)^{H}T(\mu)W_{\bot}\\
	W_{\bot}^{H}T(\mu)^{H}T(\mu)W & W_{\bot}^{H}T(\mu)^{H}T(\mu)W_{\bot}  \\
\end{pmatrix},
$$
we have $\widehat{\sigma}_{2}\geq\sigma_{2}(T(\mu))$. Then it
follows from \eqref{rR1} and \eqref{assum2} that
\begin{eqnarray*}
\widehat{\sigma}_{2}-\widehat{\sigma}_{1}&\geq&\sigma_{2}(T(\mu))-\widehat{\sigma}_{1}\\
	&\geq&\sigma_{2}(T(\lambda_{*}))
	-\|T^{\prime}(\lambda_{*})\| |\mu-\lambda_{*}|-\gamma |\mu-\lambda_*|^2\\
	&& \ \ \ -\frac{1}{2}\sigma_{2}(T(\lambda_{*}))+\|T^{\prime}(\lambda_{*})\|
	|\mu-\lambda_{*}| \\
	&\geq& \frac{1}{2}\sigma_{2}(T(\lambda_{*}))-\gamma |\mu-\lambda_*|^2>0.
\end{eqnarray*}
Therefore, $\widehat{\sigma}_{1}$ is simple, and $\widehat{x}$ is unique.
\end{proof}

Notice that, by definition, $\widehat{\sigma}_1=\|T(\mu)\widehat{x}\|$.
Relation~\eqref{rres} shows that $\widehat{\sigma}_1\rightarrow 0$ as
$\mu\rightarrow\lambda_*$, which, by Theorem~\ref{ritznep}, is met
as $\varepsilon\rightarrow 0$. Therefore, assumption \eqref{assum2} for
$\widehat{\sigma}_1$
and $\sigma_2(T(\lambda_*)$ must be fulfilled for $\varepsilon$ sufficiently
small. This theorem generalizes Theorem 2.2 of
\cite{Jia2004} to the NEP case.

We now construct an example to illustrate our convergence results on the Ritz vector
$\widetilde x$ and refined Ritz vector $\widehat x$.

\begin{example}\label{ex1}
	Consider the REP $T(\lambda)x=0$ with
$$	
T(\lambda)=\begin{pmatrix}
		\lambda & 1 & \lambda^{2} \\
		1 & \lambda & 0 \\
		0 & 0 & \frac{\lambda}{\lambda-1} \\
	\end{pmatrix}.
$$
\end{example}
Since  $\text{det}(T(\lambda))=\lambda(\lambda+1)$, it has two eigenvalues $-1$ and $0$,
 each of which has algebraic and geometric multiplicities one; see \cite{Guttel}
for the definition of algebraic and geometric multiplicities of an eigenvalue.
Suppose that we want to compute the eigenvalue $\lambda_{*}=0$ and the
associated eigenvector $x_*=(0,0,1)^T$. We generate the subspace $\mathcal{W}$
by the orthonormal matrix
$$W=\begin{pmatrix}
	0 & 1 \\
	0 & 0\\
	1 & 0 \\
\end{pmatrix},
$$
which contains the desired $x_*$ exactly, i.e., $\varepsilon=0$.
The resulting projected matrix-valued function
$$B(\lambda)=W^{H}T(\lambda)W=\begin{pmatrix}
\frac{\lambda}{\lambda-1} & 0 \\
	\lambda^{2} & \lambda \\
\end{pmatrix}.
$$
Since $\text{det}(B(\lambda))=\frac{\lambda^{2}}{\lambda-1}$, its eigenvalues,
i.e., the Ritz values, are $\mu=0$, whose
algebraic and geometric multiplicities are equal to two.
Clearly, $\mu=\lambda_{*}=0$. However, observe that
any two-dimensional nonzero vector $z$ satisfies $B(0)z=0$, which means
that any two-dimensional nonzero vector is an eigenvector $z$ of $B(0)$.
Therefore, the Ritz vector $\widetilde{x}=Wz$ is not unique, and there are
two linearly independent ones, each of which is an approximation to
$x_*$, causing the failure of the Rayleigh--Ritz method because $B(0)$ itself
does not give us any clue to which vector should be chosen. For instance, we
might choose $z=(\frac{1}{\sqrt{2}},\frac{1}{\sqrt{2}})^{T}$, in which case we
get a Ritz vector
$$
\widetilde{x}=(\frac{1}{\sqrt{2}},0,\frac{1}{\sqrt{2}})^{T},
$$
a meaningless approximation to $x_*$, and the residual norm
$$
\|T(0)\widetilde{x}\|=\frac{1}{\sqrt{2}}=O(\|T(0)\|).
$$

In contrast, the matrix
$$T(0)W=\begin{pmatrix}
	0 & 0  \\
	0  & 1 \\
	0 & 0  \\
\end{pmatrix}
$$
is column rank deficient and has rank one, and
its right singular vector $y$ with the smallest singular
value zero is unique and equals $(1,0)^T$. Therefore, the refined Ritz
vector $\widehat{x}=Wy$ is unique and $\widehat x=x_*$, and the residual norm $\|T(0)\widehat{x}\|=0$.
This example demonstrates the great superiority of the refined Ritz vector to
the Ritz vector, and shows that the refined Rayleigh--Ritz method perfectly fixes
the failure deficiency of the Rayleigh--Ritz method.

\section{Bounds for the error of the Ritz vector and the refined Ritz vector}\label{sec:bound}

We continue exploring the Ritz vector and the refined Ritz vector,
derive lower and upper bounds for $\sin\angle(\widetilde{x},\widehat{x})$,
and shed more light on these two vectors and on the implications of the bounds
for the residual norms obtained by the Rayleigh--Ritz method and the refined
Rayleigh--Ritz method.

Recall that the refined Ritz vector $\widehat{x}=Wy$ with $y$ being the
right singular vector of $T(\mu)W$ corresponding to its smallest singular
value $\widehat{\sigma}_{1}$, and keep in mind the Ritz vector
$\widetilde{x}=Wz$ and \eqref{schurB}. We can establish the following results.

\begin{theorem}\label{bound-R} With decomposition \eqref{schurB}, if $\sigma_{\min}(C(\mu))>0$, then
	\begin{equation}\label{bound-Sin}
		\frac{\widehat{\sigma}_{1}\|
(WZ_{\bot})^{H}s\|}{\sigma_{\max}(C(\mu))}\leq
		\sin\angle(\widetilde{x},\widehat{x})\leq\frac{\widehat{\sigma}_{1}\|
W^{H}s\|}{\sigma_{\min}(C(\mu))} 	
\end{equation}
with $s$ being the left singular vector of $T(\mu)W$ corresponding to its
smallest singular value $\widehat{\sigma}_{1}$.
\end{theorem}

\begin{proof}
Since $\widehat{\sigma}_{1}$ is the smallest singular value of $T(\mu)W$
and $y$ is its associated right singular vector, we have
\begin{equation}\label{By}
	B(\mu)y=W^{H}T(\mu)Wy=\widehat{\sigma}_{1}W^{H}s.
\end{equation}
By \eqref{schurB} and \eqref{By} we obtain
\begin{eqnarray}
	\begin{pmatrix}
		z^{H} \\
		Z_{\bot}^{H} \\
	\end{pmatrix}B(\mu)(
		z\  Z_{\bot})
\begin{pmatrix}
		z^{H} \\
		Z_{\bot}^{H} \\
	\end{pmatrix}y
	&=& \begin{pmatrix}
		0 & z^{H}B(\mu)Z_{\bot} \\
		0 & C(\mu) \\
	\end{pmatrix}\begin{pmatrix}
		z^{H} \\
		Z_{\bot}^{H} \\
	\end{pmatrix}y
	\nonumber \\
	&=& \begin{pmatrix}
		z^{H}B(\mu)Z_{\bot}Z_{\bot}^Hy \\
		C(\mu)Z_{\bot}^{H}y \\
	\end{pmatrix} \nonumber\\
	&=&\begin{pmatrix}
		\widehat{\sigma}_{1}z^{H}W^{H}s \\
		\widehat{\sigma}_{1}Z_{\bot}^{H}W^{H}s \\
	\end{pmatrix}, \label{lasteq}
\end{eqnarray}
respectively, where equality \eqref{lasteq} follows from \eqref{By} by left premultiplying
$B(\mu)y$ with $(z\  Z_{\bot})^H.$
Therefore,
$$
C(\mu)Z_{\bot}^{H}y=\widehat{\sigma}_{1}Z_{\bot}^{H}W^{H}s.
$$
By the orthonormality of $W$ and the above relation, we obtain
\begin{equation}\label{l-u-s}
	\sin\angle(\widetilde{x},\widehat{x})=\sin\angle(z,y)=\| Z_{\bot}^{H}y\|
=\widehat{\sigma}_{1}\| C^{-1}(\mu)(WZ_{\bot})^{H}s\|,
\end{equation}
from which it follows that
\begin{equation}\label{l-u}
	\frac{\widehat{\sigma}_{1}\| (WZ_{\bot})^{H}s\|}{\sigma_{\max}(C(\mu))}\leq
	\sin\angle(\widetilde{x},\widehat{x})\leq \frac{\widehat{\sigma}_{1}\|
(WZ_{\bot})^{H}s\|}{\sigma_{\min}(C(\mu))}\leq
	\frac{\widehat{\sigma}_{1}\| W^{H}s\|}{\sigma_{\min}(C(\mu))}.
\end{equation}
~
\end{proof}

By Theorem~\ref{con-Refined-vector}, since $\widehat{\sigma}_1\rightarrow 0$
and $\widehat{x}\rightarrow x_*$ as $\varepsilon\rightarrow 0$, the
upper bound in \eqref{bound-Sin} shows that $\sigma_{\min}(C(\mu))>0$
uniformly is a sufficient condition for $\widetilde{x}\rightarrow x_*$.
This uniform condition is in accordance with that in
Theorem~\ref{con-Ritz}, which is required to guarantee the convergence of
$\widetilde{x}$.
The lower bound tends to zero provided that $\widehat{\sigma}_1\rightarrow 0$.
Therefore, generally  $\sin\angle(\widetilde{x},\widehat{x})=0$ if and only if
$\widehat{\sigma}_{1}=0$; in other words, if $\widehat{x}\not=x_*$,   then
$\widetilde{x}\neq\widehat{x}$ generally. This theorem is an extension of
Theorem 3.1 in \cite{Jia2004}.

Denote $\widetilde{r}=T(\mu)\widetilde{x}$ and $\widehat{r}=T(\mu)\widehat{x}$.
By definition, we have $\|\widehat{r}\|=\widehat{\sigma}_{1}$
and $\|\widehat{r}\|\leq\|\widetilde{r}\|$.
The following theorem gets more insight into these two residual norms.

\begin{theorem}\label{normR}
	Let $\widehat{\sigma}_{1}$ and $\widehat{\sigma}_{m}$ be the smallest and
	largest singular values of $T(\mu) W$. The following results hold:
	\begin{equation}\label{rr1}
		\cos^{2}\angle(\widetilde{x},\widehat{x})+	\bigg(\frac{\widehat{\sigma}_{2}}
{\widehat{\sigma}_{1}}\bigg)^{2}
\sin^{2}\angle(\widetilde{x},\widehat{x})\leq\frac{\| \widetilde{r}\|^{2}}{\|
\widehat{r}\|^{2}}\leq\cos^{2}\angle(\widetilde{x},\widehat{x})+
\bigg(\frac{\widehat{\sigma}_{m}}{\widehat{\sigma}_{1}}\bigg)^{2}
\sin^{2}\angle(\widetilde{x},\widehat{x}).
	\end{equation}
\end{theorem}

\begin{proof}
Note $\| \widehat{r}\|=\widehat{\sigma}_{1}$. The proof is similar to that of Theorem 4.1 in \cite{Jia2004} and is thus omitted.
\end{proof}

The second terms of the lower and upper bounds in
\eqref{rr1} play a key role in deciding the size of $\|\widetilde{r}\|/
\| \widehat{r}\|$. We may have $\|\widehat{r}\|\ll\|\widetilde{r}\|$
because the lower and upper bounds for $\|\widetilde{r}\|/\| \widehat{r}\|$
in (\ref{rr1}) may be much bigger than one. This is because
$\sin\angle(\widetilde{x},\widehat{x})$
may tend to zero much more slowly than $\widehat{\sigma}_{1}$ and
it may be not small and even may be arbitrarily close to one,
so that the lower and upper bounds are arbitrarily large.
As a matter of fact, in the previous example, we have
$\widehat{\sigma}_1=\|\widehat{r}\|=0$, and both
the lower and upper bounds are thus infinite.

Let us further explore Example~\ref{ex1}. In practice it is rare that
$\varepsilon=0$. However, a small $\varepsilon$ may still make the
Rayleigh--Ritz method fail.
We show that the Ritz vector $\widetilde{x}$ may have no accuracy
no matter how small $\varepsilon$ is, while the refined Rayleigh--Ritz
method works perfectly.

\begin{example} \label{ex2}
In Example~\ref{ex1}, perturb $W$ by a random normal deviation matrix
with the standard deviation $10^{-4}$, and orthonormalize $W$.
Then $\varepsilon=O(10^{-4})$, and
the projected matrix-valued function
$$
B(\lambda)=\frac{1}{\lambda-1}\begin{pmatrix}
	b_{11}(\lambda) & b_{12}(\lambda) \\
	b_{21}(\lambda) & b_{22}(\lambda) \\
\end{pmatrix}
$$
with
\begin{eqnarray*}
b_{11}(\lambda)&=&8.8849\times 10^{-5}\lambda^{3}-8.8828\times 10^{-5}\lambda^{2}+
\lambda+2.0385\times 10^{-8},
\\
b_{12}(\lambda)&=&7.8972\times 10^{-9}\lambda^{3}-8.8891\times 10^{-5}
\lambda^{2}+2.9251\times 10^{-4}\lambda-1.1475\times 10^{-4},
\\
b_{21}(\lambda)&=&-\lambda^{3}+\lambda^{2}+2.9251\times 10^{-4}\lambda-1.1475\times 10^{-4},
\\
b_{22}(\lambda)&=&-8.8883\times 10^{-5}\lambda^{3}+1.0001\lambda^{2}-1.0006\lambda
+5.8885\times 10^{-4}.
\end{eqnarray*}

Notice that, in the numerator of $\det(B(\lambda))=0$,
the coefficients of $\lambda^6$ and $\lambda^5$ are exactly zero.
The eigenvalues of $B(\lambda)$ are
$$
7.3993\times 10^{3},\ -4.0016\times 10^{3},\ 5.6570\times 10^{-4}, \ 2.3256\times 10^{-5}.
$$
According to Algorithm~\ref{alg-RR}, we take the smallest $2.3256\times 10^{-5}$ in
magnitude to approximate $\lambda_*=0$, whose error
is $O(\varepsilon)$ and confirms that the error bound \eqref{ritzbound} in Theorem~\ref{ritznep}
is sharp where $m_{\mu}=1$. The associated Ritz vector
$$
\widetilde{x}=(1.9873\times 10^{-1},5.3892\times 10^{-5},-9.8005\times 10^{-1})^{T},
$$
is an approximation to the desired $x_*$ with little accuracy
since
$$
\sin\angle(\widetilde{x},x_*)=1.9873\times 10^{-1}=O(1)
$$
and
$$
\|\widetilde{r}\|=\|T(\mu)\widetilde{x}\|=1.9873\times 10^{-1}=O(\|T(\mu)\|).
$$
In contrast, the refined Ritz vector
$$
\widehat{x}=(-2.8435\times 10^{-8},-1.1469\times 10^{-4},1)^{T},
$$
and
\begin{eqnarray*}
\sin\angle(\widehat{x},x_*)&=&1.1469\times 10^{-4}=O(\varepsilon),\\
\|\widehat{r}\|=\|T(\mu)\widehat{x}\|&=&1.1703\times 10^{-4}=O(\varepsilon).
\end{eqnarray*}
Therefore, $\widehat{x}$ is an excellent approximation to $x_*$.
Furthermore,
$$
\sin\angle(\widetilde{x},\widehat{x})=1.9872\times 10^{-1}=O(1)
$$
and the ratio
$$
\dfrac{\|\widehat{r}\|}{\|\widetilde{r}\|}=5.8887\times 10^{-4}=O(\varepsilon).
$$
Note that $\widehat{\sigma}_{1}=1.1703\times 10^{-4}$. Then we see that the lower and
upper bounds in (\ref{l-u}) are very close and they estimate the ratio very
accurately.

As a matter of fact, for this example, if we perturb $W$ by a random normal
deviation matrix with an arbitrarily small standard deviation
$\varepsilon$  or, more generally, by a general perturbation of size $O(\varepsilon)$,
then, by the continuity of eigenvalues of
$B(\lambda)$ in its elements, there are always two Ritz values that are simple but
close to the simple eigenvalue $\lambda_*=0$ with errors $O(\varepsilon)$. Therefore,
$\sigma_{\min}(C(\lambda_*))\rightarrow 0$ and $\sigma_{\min}(L(\mu))=O(\|T(\mu)\|)$
 in bounds \eqref{sin-x1} and \eqref{errorrefine} as $\varepsilon\rightarrow 0$,
respectively. Consequently, bound \eqref{sin-x1} is much bigger than $O(\varepsilon)$
and may be $O(1)$, bound \eqref{errorrefine} and
bound~\eqref{rres} for the residual norm $\|\widehat{r}\|$ are definitely $O(\varepsilon)$,
indicating that we may have
$$
O(\varepsilon)\ll \sin\angle(\widetilde{x},x_*)=O(1) \mbox{ and }
O(\varepsilon)\ll \|\widetilde{r}\|=O(\|T(\mu)\|)
$$
but definitely have
$$
\sin\angle(\widehat{x},x_*)\leq O(\varepsilon)\ \mbox{ and }
\|\widehat{r}\|\leq O(\varepsilon).
$$
Therefore, the refined Rayleigh--Ritz method works excellently and may work much better
than the Rayleigh--Ritz method; the Ritz vector $\widetilde x$ may be very inaccurate and
even may have no accuracy even if it is unique, and the corresponding
residual norm $\|\widetilde{r}\|$
not only may not tend to zero but also is not small.

In what follows we numerically confirm the above claims by taking the four
$\varepsilon=10^{-3},10^{-4},10^{-6},10^{-8}$. For each $\varepsilon$,
we generate 100 random matrices with the standard deviation $\varepsilon$.
For the Rayleigh--Ritz method, we compute the 100 errors $\sin\angle(\widetilde{x},x_*)$
and residual norms $\|\widetilde{r}\|$, record their maximum, minimum and mean values.
For the refined Rayleigh--Ritz method, we do the same things. Table~\ref{table1}
reports the results obtained;  for each of the four $\varepsilon$,
Figure~\ref{fig1} depicts the 100 values of
$\sin\angle(\widetilde{x},x_*)$'s, $\|\widetilde{r}\|$'s and $\sin\angle(\widehat{x},x_*)$'s,
$\|\widehat{r}\|$'s.

\begin{table}[htbp]  \label{table1}
	\caption{The minimum, maximum, and mean values of $\sin\angle(\widetilde{x},x_*)$'s,
$\|\widetilde{r}\|$'s, $\sin\angle(\widehat{x},x_*)$'s and $\|\widehat{r}\|$'s.
		\label{tab:its0}}
	\begin{center}
		\tabcolsep5mm
		\def\arraystretch{1.3}
		\small
		\begin{tabular}{c|c|c|c}
			\hline
	$\varepsilon=10^{-3}$	& minimum&maximum& mean
			\tabularnewline
			\hline
$\sin\angle(\widetilde{x},x_*)$ &  $3.1025\times10^{-3}$ &  $7.0742\times10^{-1}$  &  $4.5453\times10^{-1}$   \tabularnewline
$\|\widetilde{r}\|$&    $3.2070\times10^{-3}$ &  $7.0742\times10^{-1}$  &  $4.5453\times10^{-1}$   \tabularnewline
$\sin\angle(\widehat{x},x_*)$ &  $1.8608\times10^{-6}$ &  $2.5848\times10^{-3}$  &  $7.8344\times10^{-4}$   \tabularnewline
$\|\widehat{r}\|$ &  $2.0850\times10^{-6}$ &  $7.0711\times10^{-3}$  &  $4.5455\times10^{-3}$   \tabularnewline
			\hline
						\hline
			$\varepsilon=10^{-4}$	& minimum&maximum& mean
 			\tabularnewline
			\hline
			$\sin\angle(\widetilde{x},x_*)$ &  $2.1057\times10^{-3}$ &  $7.0714\times10^{-1}$  &  $4.5455\times10^{-1}$   \tabularnewline
			$\|\widetilde{r}\|$&    $2.1057\times10^{-3}$ &  $7.0714\times10^{-1}$  &  $4.5455\times10^{-1}$   \tabularnewline
			$\sin\angle(\widehat{x},x_*)$ &  $1.7069\times10^{-7}$ &  $2.5880\times10^{-4}$  &  $7.8372\times10^{-5}$   \tabularnewline
			$\|\widehat{r}\|$ &  $1.7069\times10^{-7}$ &  $3.6592\times10^{-4}$  &  $1.2777\times10^{-4}$   \tabularnewline
			\hline
							\hline
			$\varepsilon=10^{-6}$	& minimum&maximum& mean
			\tabularnewline
			\hline
			$\sin\angle(\widetilde{x},x_*)$ &  $2.0849\times10^{-3}$ &  $7.0711\times10^{-1}$  &  $4.5455\times10^{-1}$   \tabularnewline
			$\|\widetilde{r}\|$&    $2.0850\times10^{-3}$ &  $7.0711\times10^{-1}$  &  $4.5455\times10^{-1}$   \tabularnewline
			$\sin\angle(\widehat{x},x_*)$ &  $1.6900\times10^{-9}$ &  $2.5883\times10^{-6}$  &  $7.8375\times10^{-7}$   \tabularnewline
			$\|\widehat{r}\|$ &  $1.6900\times10^{-9}$ &  $3.7936\times10^{-6}$  &  $1.3725\times10^{-6}$   \tabularnewline
			\hline
								\hline
			$\varepsilon=10^{-8}$	& minimum&maximum& mean
			\tabularnewline
			\hline
			$\sin\angle(\widetilde{x},x_*)$ &  $6.8201\times10^{-3}$ &  $1.0000\times10^{0}$  &  $4.6453\times10^{-1}$   \tabularnewline
			$\|\widetilde{r}\|$&    $6.8208\times10^{-3}$ &  $1.0000\times10^{0}$  &  $4.6453\times10^{-1}$   \tabularnewline
			$\sin\angle(\widehat{x},x_*)$ &  $1.6898\times10^{-11}$ &  $2.5883\times10^{-8}$  &  $7.8375\times10^{-9}$   \tabularnewline
			$\|\widehat{r}\|$ &  $1.6898\times10^{-11}$ &  $3.6604\times10^{-8}$  &  $1.2472\times10^{-8}$   \tabularnewline
			\hline
		\end{tabular}
	\end{center}
\end{table}

For the Rayleigh--Ritz method, Table~\ref{table1} confirms our claims. As we see, the maximum and
mean values of $\sin\angle(\widetilde{x},x_*)$'s and $\|\widetilde{r}\|$'s are all $O(1)$
and remain unchanged essentially for the four greatly varying $\varepsilon$; their minima $O(10^{-3})$
are almost independent of the size of $\varepsilon$ and are {\em several orders bigger} than
$\varepsilon=10^{-6},10^{-8}$, showing that the $\widetilde x$ are poor approximations to
$x_*$. Furthermore, for each $\varepsilon$, keeping in mind the minimum $\sin\angle(\widetilde{x},x_*)$
and $\|\widetilde{r}\|$, we find that
the errors of Ritz vectors $\widetilde{x}$ attain $\sin\angle(\widetilde{x},x_*)=O(10^{-3})$
and $\|\widetilde{r}\|=O(10^{-3})$ only for six ones among the 100 test matrices, as
is observed from the upper ones of Figure \ref{fig1} (a)-(d).
Therefore, the Rayleigh--Ritz method really works poorly for this example.

In contrast, the refined Rayleigh--Ritz method works perfectly. As Table~\ref{table1} indicates clearly,
for the given four $\varepsilon$, the maximum error $\sin\angle(\widehat{x},x_*)=O(\varepsilon)$ and
the maximum residual norm $\|\widehat{r}\|=O(\varepsilon)$ always. The lower ones of Figure
\ref{fig1} (a)-(d) exhibit the 100 errors $\sin\angle(\widehat{x},x_*)$ and the 100 residual norms
$\|\widehat{r}\|$.

\begin{figure}[htbp]  	
\centering\tabcolsep=0mm
	{\footnotesize{\doublerulesep0.5pt
			\begin{tabular}
				[c]{cc}%
				\includegraphics[height=1.8in]{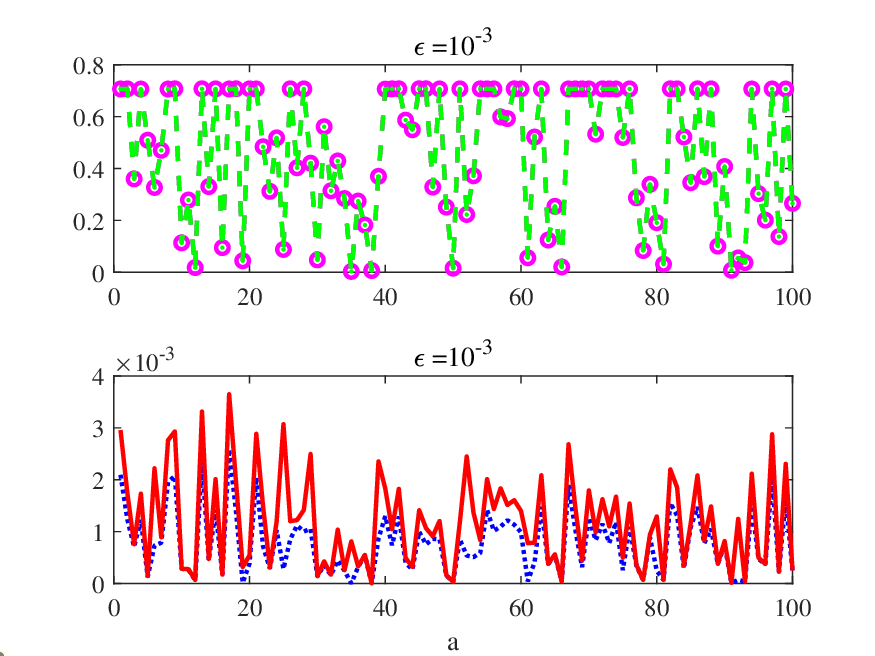}&
				\includegraphics[height=1.8in]{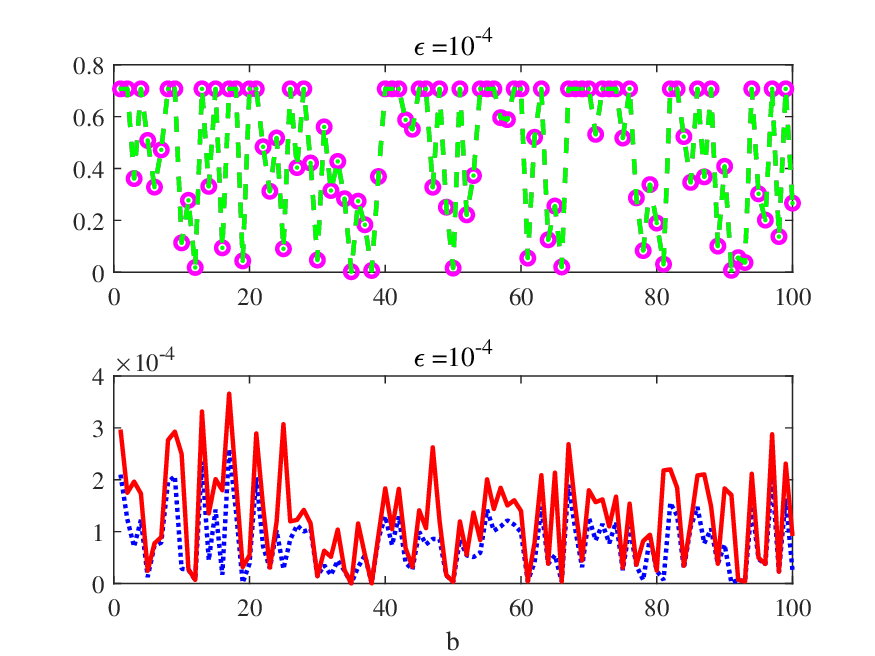}\\
				\includegraphics[height=1.8in]{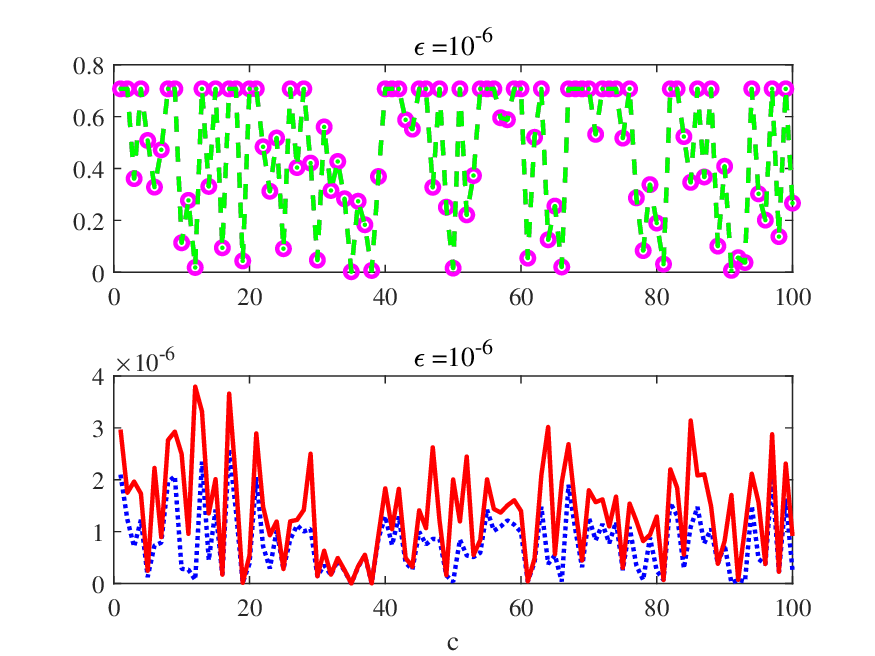}&
				\includegraphics[height=1.8in]{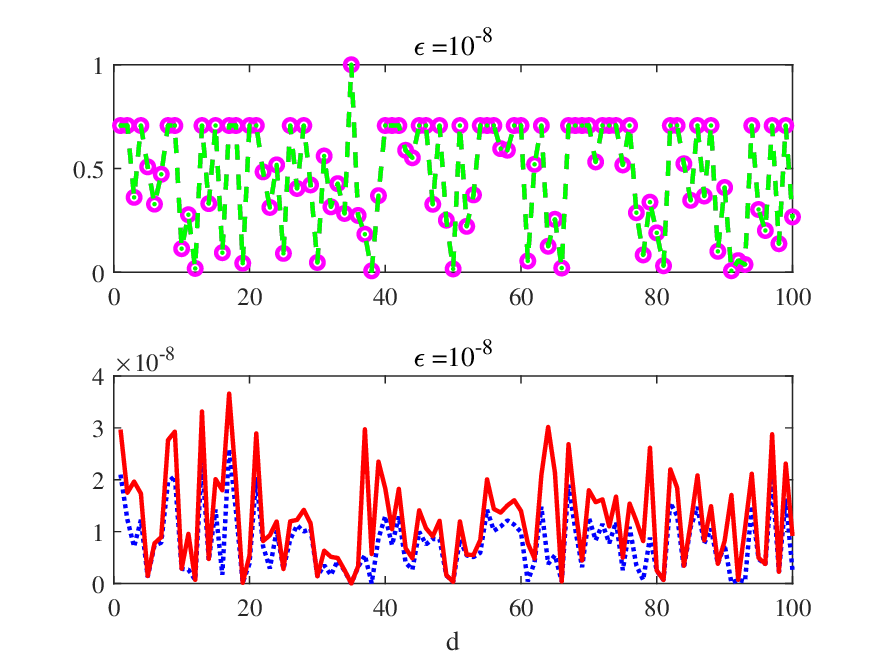}
	\end{tabular}}}
	\caption{Curves of $\sin\angle(\widetilde{x},x_*)$, $\|\widetilde{r}\|$
and $\sin\angle(\widehat{x},x_*)$,
$\|\widehat{r}\|$ for  with each of the four $\varepsilon$ and 100 random perturbation matrices,
and the upper and lower figures in (a)-(d)
  are for $\sin\angle(\widetilde{x},x_*)$, $\|\widetilde{r}\|$ and $\sin\angle(\widehat{x},x_*)$,
  $\|\widehat{r}\|$, respectively, where the `-o' line is the curve of $\sin\angle(\widetilde{x},x_*)$, the
  dashed dot line is the curve of $\|\widetilde{r}\|$, the dotted line is the curve of
  $\sin\angle(\widehat{x},x_*)$, and the solid line is the curve of $\|\widehat{r}\|$.}
	\label{fig1}
\end{figure}

\end{example}

\section{Concluding remarks}\label{sec:concl}

We have studied the convergence of the Rayleigh--Ritz method and the refined
Rayleigh--Ritz method for computing a simple eigenpair $(\lambda_*,x_*)$ of NEP \eqref{NEP}.
In terms of the deviation $\varepsilon$ of
$x_*$ from a given subspace $\mathcal{W}$, we have
established an a priori error estimate for the Ritz value $\mu$ and error bounds for
the Ritz vector $\widetilde{x}$ and the refined Ritz vector $\widehat{x}$, respectively.

We have also derived lower and upper bounds for the error of the
Ritz vector and the refined Ritz vector as well as for the ratio of the
residual norms obtained by the two methods. In the meantime, we
have presented an approach to decide the convergence of an arbitrary
approximate eigenvector, which includes the Ritz vector and the refined Ritz
vector, by checking the computable residual norm. The obtained
bound plays a crucial role in establishing the a priori error bounds for the refined
Ritz vector and the residual norm of the refined Ritz pair.

Our conclusions are that, as $\varepsilon\rightarrow 0$,
(i) there exists a Ritz value $\mu$ that converges to
$\lambda_*$ unconditionally but the Ritz vector $\widetilde{x}$
converges conditionally and it even may not be unique, causing
the failure of the Rayleigh--Ritz method, and (ii) the refined Ritz vector
$\widehat{x}$ is unique and converges unconditionally. Therefore,
the refined Rayleigh--Ritz method is more robust and has better convergence
than the Rayleigh--Ritz method. Our results have nontrivially
extended some of the main results
in \cite{Jia2004,Jia1999,Jia2001} for the linear eigenvalue problem to the NEP.
The results provide necessary theoretical supports for the two kinds of
projection methods for numerical solutions of large NEPs.

We have constructed examples to confirm our theoretical results.
They have illustrated that the refined Ritz vector is unique and is much
more accurate than the Ritz vector and that there are indeed more than one
Ritz vectors to approximate the unique desired eigenvector. If
the Ritz vector is unique, it may have no accuracy, and the residual norm by the
refined Rayleigh--Ritz method is much smaller than that by the
Rayleigh--Ritz method.

There remain some important issues and problems to be considered.
As we have addressed throughout, all the results in this paper
are established under the assumption that $\lambda_*$ is
simple. For the linear eigenvalue problem, this assumption
has been dropped in \cite{Jia2001}, which includes more general
convergence results on the Rayleigh--Ritz method for computing
more than one eigenvalues and the associated eigenspace or
invariant subspace, named as
an eigenblock. For the NEP, one can define invariant pairs \cite{Guttel}, which are
generalizations of eigenblocks.
Effenberger \cite{effenberger2013} proposes a nonlinear Jacobi--Davidson
type method with deflation strategy to successively compute an invariant pair.
Suppose that a given subspace $\mathcal{W}$ contains sufficiently
approximations to the desired invariant subspace. Then how to
extend the convergence results in \cite{Jia2001} and this paper to the
method for the NEP? We will explore these problems and issues in future work.

\section*{Declarations}

The two authors declare that they have no
financial interests, and they read and approved the final manuscript.

\section*{Acknowledgements}
We thank the three reviewers and editor Professor Wen-Wei Lin for their very careful reading of the paper
and for their valuable comments and suggestions, which made us improve the presentation considerably.


\end{document}